\documentclass[10pt,reqno]{amsart}

\usepackage{color}
\usepackage{enumerate}
\usepackage{amssymb}
\usepackage{amsmath}
\usepackage{amsthm}
\usepackage{mathrsfs}

\newtheorem{theorem}{Theorem}[section]
\newtheorem{lemma}[theorem]{Lemma}

\theoremstyle{remark}
\newtheorem{remark}[theorem]{Remark}

\theoremstyle{definition}
\newtheorem{assumption}[theorem]{Assumption}

\newtheorem{definition}[theorem]{Definition}

\newcommand\cbrk{\text{$]$\kern-.15em$]$}}
\newcommand\opar{\text{\,\raise.2ex\hbox{${\scriptstyle|}$}\kern-.34em$($}}
\newcommand\cpar{\text{$)$\kern-.34em\raise.2ex\hbox{${\scriptstyle |}$}}\,}

\def\qed{{\hfill $\Box$ \bigskip}}

\def\XXint#1#2#3{{\setbox0=\hbox{$#1{#2#3}{\int}$}
\vcenter{\hbox{$#2#3$}}\kern-.5\wd0}}

\newcommand\bB{\mathbb{B}}
\newcommand\bL{\mathbb{L}}
\newcommand\bH{\mathbb{H}}
\newcommand\bZ{\mathbb{Z}}

\newcommand\bE{\mathbb{E}}

\newcommand\bN{\mathbb{N}}

\newcommand\fR{\mathbf{R}}

\newcommand\cB{\mathcal{B}}

\newcommand\cD{\mathcal{D}}
\newcommand\cF{\mathcal{F}}

\newcommand\cP{\mathcal{P}}

\newcommand\cS{\mathcal{S}}
\newcommand\cM{\mathcal{M}}

\newcommand\rF{\mathscr{F}}

\newcommand{\mysection}[1]{\section{#1}
\setcounter{equation}{0}}

\begin{document}

\title[SPDE with degenerate and unbounded coefficients]
{A sharp $L_p$-regularity result for  second-order stochastic partial differential equations with unbounded and fully degenerate leading coefficients}

\author{Ildoo Kim}
\address{Department of mathematics, Korea university, 1 anam-dong
sungbuk-gu, Seoul, south Korea 136-701}
\email{waldoo@korea.ac.kr}
\thanks{I. Kim has been supported by the National Research Foundation of Korea(NRF) grant funded by the Korea government(MSIP) (No. 2017R1C1B1002830).}

\author{Kyeong-Hun Kim}
\address{Department of mathematics, Korea university, 1 anam-dong
sungbuk-gu, Seoul, south Korea 136-701}
\email{kyeonghun@korea.ac.kr}

\thanks{}



\begin{abstract}
  We present existence, uniqueness, and sharp regularity results of solution 
 to the stochastic partial differential equation (SPDE)
 \begin{align}
					\label{abs eqn}
 du=(a^{ij}(\omega,t)u_{x^ix^j}+f)dt + (\sigma^{ik}(\omega,t)u_{x^i}+g^k)dw^k_t, \quad u(0,x)=u_0,
 \end{align}
 where  $\{w^k_t:k=1,2,\cdots\}$ is a sequence of independent Brownian motions.  The coefficients  are merely measurable in $(\omega,t)$ and can be unbounded and  fully degenerate, that is, 
coefficients $a^{ij}$, $\sigma^{ik}$  merely satisfy 
 \begin{align}
					\label{abs only}
\left(\alpha^{ij}(\omega,t)\right)_{d\times d}:= \left(a^{ij}(\omega,t)-\frac{1}{2}\sum_{k=1}^{\infty} \sigma^{ik}(\omega,t)\sigma^{jk}(\omega,t)\right) \geq 0.
 \end{align}
In this article, we prove that there exists a unique solution $u$ to \eqref{abs eqn}, and 
\begin{align}
							\notag
\|u_{xx}\|_{\bH^\gamma_p(\tau,\delta)}   
&\leq N(d,p) \bigg(  \|u_0\|_{\bB_p^{\gamma+2 \left(1-1/ p \right)}}  +  \|  f\|_{\bH^\gamma_p( \tau,\delta^{1-p} )}   \\
								\label{abs est}
&\qquad \qquad+\|g_x\|^p_{\bH^\gamma_p( \tau, |\sigma|^p \delta^{1-p},l_2)}+ \|  g_x\|_{\bH^\gamma_p( \tau,\delta^{1-p/2},l_2)} \bigg),
\end{align}
where $p\geq 2$, $\gamma\in \fR$, $\tau$ is an arbitrary stopping time, $\delta(\omega, t)$ is the smallest eigenvalue of $\alpha^{ij}(\omega, t)$, $\bH_p^\gamma(\tau, \delta)$ is a weighted stochastic Sobolev space, and $\bB_p^{\gamma+2 \left(1-1/ p \right)}$ is a stochastic Besov space.
\end{abstract}

\maketitle

\mysection{introduction}

The second-order elliptic and parabolic partial differential equations (PDEs) with unbounded or degenerate leading coefficients have been widely studied for a long time (see .e.g. \cite{LL,O,O1,O2,O3}).   
Such equations  naturally arise   in the modeling of random phenomenon related to   diffusion.  For instance, consider the stochastic process $X_t$ governed by
$$
dX_t= b(\omega,t)dt+ \sigma(\omega,t)dB_t, \quad X_0=x,
$$
where $b(\omega,t)$ is $\fR^d$-valued, $\sigma(\omega,t)$ is $d\times d$-matrix-valued, and $B_t$ is $d$-dimensional Brownian motion. 
Then,  for any smooth function $f(x)$, $u(t,x):=\bE\left[ f(X_t) \right]$ satisfies the parabolic PDE
$$
u_t=\frac{1}{2}(\sigma \sigma^*)^{ij}u_{x^ix^j}+b^iu_{x^i}, \quad t>0; \quad u(0,\cdot)=f(x).
$$
Here $\sigma^*$ is the transpose of $\sigma$.  Since $(\sigma \sigma^*)$ is symmetric, it is only guaranteed that 
\begin{equation}
  \label{elliptic}
(a^{ij}(\omega,t)):=(\sigma \sigma^*) \geq 0.
\end{equation}
Such connections between PDEs and stochastic processes  illustrate  that boundedness and uniform ellipticity conditions of leading coefficients  are somewhat restrictive for the study of general PDEs (and SPDEs).

In this article we study  a weighted $L_p$-regularity theory ($p\geq 2$) of SPDE
\begin{align}
					\label{abs eqn-1}
 du=(a^{ij}(\omega,t)u_{x^ix^j}+f)dt + (\sigma^{ik}(\omega,t)u_{x^i}+g^k)dw^k_t, \quad u(0,x)=u_0,
 \end{align}
 where indices $i,j$ moves from $1$ to $d$, $k$ runs through $\{1,2,3,\cdots\}$. Einstein's summation convention with respect to repeated indices $i,j,k$ is assumed.
We assume very minimal conditions on the coefficients, that is,  the coefficients are merely measurable in $(\omega,t)$ and satisfy 
\begin{align}
					\label{abs only-1}
\left(\alpha^{ij}(\omega,t)\right):= \left(a^{ij}(\omega,t)-\frac{1}{2}\sum_{k=1}^{\infty} \sigma^{ik}(\omega,t)\sigma^{jk}(\omega,t)\right) \geq 0,
 \end{align}
 together with the local integrability
\begin{equation}
  \label{local}
\int^t_0 |a^{ij}(\omega,s)|ds +\int^t_0  \sum_{k=1}^\infty \left|\sigma^{ik}(\omega,s) \right|^2 ds<\infty, \quad (a.s.) \quad \forall \, t>0,~i,j.
\end{equation}
Actually  condition \eqref{local} is necessary  to make sense of equation \eqref{abs eqn-1}.

To the best of our knowledge, the theory of SPDE with  degenerate and unbounded leading coefficients  was initiated in \cite{krylov1986characteristics} and \cite{gyongy1990stochastic} respectively, and the result in \cite{krylov1986characteristics} was extended to the case of system in \cite{gerencser2015solvability}. 
Recently, this type SPDEs have been developed in various directions in $L_2$-spaces.
For instance, a regular strong solution to quasilinear degenerate SPDEs is studied in \cite{gess2017stochastic}
and the  existence of  an $L_2$-valued continuous solution to SPDEs with space-time dependent random coefficients $a^{ij}(t,x)$ which are allowed to  be both unbounded and degenerate is handled in \cite{zhang2011stochastic}.

However, 
 roughly speaking, if $p>2$ and the coefficients are degenerate then  the  results in the literature  (see e.g. \cite{krylov1986characteristics}) only say that  equation \eqref{abs eqn-1} has a unique  continuous $L_p$-valued solution $u$ and 
 \begin{align}
						\label{eqn 1}
 \bE \sup_{t\leq T} \|u\|^p_{L_p}\leq N(d,p,T) \left( \bE \|u_0\|^p_{L_p} + \|f\|^p_{\bL_p(T)}+\||g|_{l_2}\|^p_{\bL_p(T)} \right),
 \end{align}
 where $\bL_p(T)=L_p(\Omega\times [0,T], L_p(\fR^d))$. 
 Note that in  estimate \eqref{eqn 1} the solution $u$ is not smoother than $u_0, f$ and $g$. Actually \eqref{eqn 1}  is the best possible estimation in the
 extreme degenerate case, i.e. if  $a^{ij}(\omega,t)=0$ and $\sigma^{ik}(\omega,t)=0$ for all $i$, $j$, $k$, $\omega$, $t$. Because, in this case,  we have
$$
u(t,x) =u_0(x)+ \int_0^t f(s,x)ds + \int_0^t g^k(s,x)dw^k_s,
$$
and thus it  cannot be expected that the solution $u$ is smoother than data $u_0$, $f$, and $g$.
In other words, if degeneracy of diffusion is too strong, then there is no smoothing effect enough to make solutions regular than data.
However, if the matrix $(\alpha^{ij}(t))_{d\times d}$ in \eqref{abs only-1} is not identically zero then the question whether $u_{xx}\in L_p(\fR^d)$ on the set $\{(\omega,t): (\alpha^{ij})_{d\times d} >0\}$ naturally arises.   

It turns out the the answer to the above question is ``yes". In this article we prove that  under the conditions \eqref{abs only-1} and \eqref{local}, it holds that 
\begin{align}
							\notag
\bE \int^{\tau}_0 \|u_{xx}\|^p_{H^\gamma_p}\delta\,dt 
&\leq N(d,p) \bigg( \bE\|u_0\|^p_{B_p^{\gamma+2 \left(1-1/ p \right)}}  + \bE \int^{\tau}_0 \|f\|^p_{H^\gamma_p}\delta^{1-p}\,dt  \\
								\label{abs est-1}
&+\bE \int^{\tau}_0 \|g_x\|^p_{H^\gamma_p(l_2)} ( |\sigma|^p \delta^{1-p} + \delta^{1-p/2})dt \bigg),
\end{align}
where $p\geq 2$, $\gamma\in \fR$,  $\tau$ is an arbitrary stopping time, $\delta=\delta(\omega, t)$ is the smallest eigenvalue of $(\alpha^{ij}(\omega, t))$, $|\sigma(\omega,t)|=\max_{i}|\sigma^i|_{l_2}$, $H_p^\gamma$ is a  Sobolev space, and $B_p^{\gamma+2 \left(1-1/ p \right)}$ is a  Besov space.

We mention that our weights in   estimate \eqref{abs est-1} are not  in $A_p$-weight class which is a very important function class in the Fourier analysis (see Remark \ref{ap rem} below). 
Thus, even if the coefficients are not random, estimate \eqref{abs est-1} can not be obtained  based on the estimation of the sharp function $(u_{xx})^{\#}$ or Calder\'on-Zygmund approach. 
See e.g. \cite{Kr01, kim2016singular} for detail of such approachs. 

In summary, we list the novelty of our result:
\begin{enumerate}
\item Coefficients $a^{ij}(\omega, t)$ and $\delta^{ik}(\omega, t)$ are not necessarily bounded.

\item  The matrix $(\alpha^{ij})_{d\times d} (\omega,t):= \left(a^{ij}(\omega,t)-\frac{1}{2}\sum_{k=1}^{\infty} \sigma^{ik}(\omega,t)\sigma^{jk}(\omega,t)\right) $ can be  (fully) degenerate.

\item  Coefficients $a^{ij}(\omega, t)$ and $\delta^{ik}(\omega, t)$ can be random, and merely measurable in $(\omega,t)$.

\item  A sharp weighted $L_p$-regularity result is obtained   for any $p\geq 2$.


\end{enumerate}

\smallskip

This article is organized as follows. In 
Section 2, we introduce our  main result  together with some related function spaces. 
In Section 3, we prove the solvability and a priori  estimate for deterministic equations without boundedness and ellipticity conditions on the leading coefficients. 
In Section 4, stochastic PDEs with additive noises  are treated, and finally  the proof of the main theorem is given  in Section 5.

\smallskip

We finish the introduction with  notation used in the article. 
\begin{itemize}
\item $\bN$ and $\bZ$ denote the natural number system and the integer number system, respectively.
As usual $\fR^{d}$
stands for the Euclidean space of points $x=(x^{1},...,x^{d})$,
$\fR^d_+:=\{x=(x^1,\cdots,x^d)\in \fR^d: x^1>0\}$ and
$B_r(x):=\{y\in \fR^d: |x-y|<r\}$.
 For $i=1,...,d$, multi-indices $\alpha=(\alpha_{1},...,\alpha_{d})$,
$\alpha_{i}\in\{0,1,2,...\}$, and functions $u(x)$ we set
$$
u_{x^{i}}=\frac{\partial u}{\partial x^{i}}=D_{i}u,\quad
D^{\alpha}u=D_{1}^{\alpha_{1}}\cdot...\cdot D^{\alpha_{d}}_{d}u,
\quad  \nabla u=(u_{x^1}, u_{x^2}, \cdots, u_{x^d}).
$$
We also use the notation $D^m$ for a partial derivative of order $m$ with respect to $x$.

\item $C^\infty(\fR^d)$ denotes the space of infinitely differentiable functions on $\fR^d$. 
$\cS(\fR^d)$ is the Schwartz space consisting of infinitely differentiable and rapidly decreasing functions on $\fR^d$.
By $C_c^\infty(\fR^d)$,  we denote the subspace of $C^\infty(\fR^d)$ consisting of functions with compact support.

\item For $p \in [1,\infty)$, a normed space $F$ 
and a  measure space $(X,\mathcal{M},\mu)$, $L_{p}(X,\cM,\mu;F)$
denotes the space of all $F$-valued $\mathcal{M}^{\mu}$-measurable functions
$u$ so that
\[
\left\Vert u\right\Vert _{L_{p}(X,\cM,\mu;F)}:=\left(\int_{X}\left\Vert u(x)\right\Vert _{F}^{p}\mu(dx)\right)^{1/p}<\infty,
\]
where $\mathcal{M}^{\mu}$ denotes the completion of $\cM$ with respect to the measure $\mu$. 
We write $u \in L_{\infty}(X,\cM,\mu;F)$ iff
$$
\sup_{x}|u(x)| := \|u\|_{L_{\infty}(X,\cM,\mu;F)} 
:= \inf\left\{ \nu \geq 0 : \mu( \{ x: \|u(x)\|_F > \nu\})=0\right\} <\infty.
$$
If there is no confusion for the given measure and $\sigma$-algebra, we usually omit the measure and the $\sigma$-algebra. 
Moreover, if a topology is given on $X$, then the subspace of all continuous functions in $ L_{\infty}(X,\cM,\mu;F)$ is denoted by
$C(X;F)$.

\item  For functions  depending on $\omega$, $t$,  and $x$, the random parameter $\omega \in \Omega$ is  usually omitted. 

\item By $\cF$ and $\cF^{-1}$ we denote the d-dimensional Fourier transform and the inverse Fourier transform, respectively. That is,
$\cF[f](\xi) := \int_{\fR^{d}} e^{-i x \cdot \xi} f(x) dx$ and $\cF^{-1}[f](x) := \frac{1}{(2\pi)^d}\int_{\fR^{d}} e^{ i\xi \cdot x} f(\xi) d\xi$.

\item 
If we write $N=N(a,b,\cdots)$, this means that the
constant $N$ depends only on $a,b,\cdots$. 
\end{itemize}




\mysection{Setting and main results}

Let $(\Omega,\rF,P)$ be a complete probability space,
$\{\rF_{t},t\geq0\}$ be an increasing filtration of
$\sigma$-fields $\rF_{t}\subset\rF$, each of which contains all
$(\rF,P)$-null sets. By  $\cP$ we denote the predictable
$\sigma$-algebra generated by $\{\rF_{t},t\geq0\}$ and  we assume that
on $\Omega$ there exist  independent one-dimensional Wiener
processes $w^{1}_{t},w^{2}_{t},...$, each of which is a Wiener
process relative to $\{\rF_{t},t\geq0\}$.

We study the following initial value problem on $\fR^d$:
\begin{align}
				\label{main eqn}
				 du= \left( a^{ij}(t)u_{x^ix^j}+ f  \right)dt + \left(\sigma^{ik}(t)u_{x^i}+g^k\right)dw^k_t, \quad t>0; \quad u(0,\cdot)=u_0. 
\end{align}
As mentioned in the introduction,  Einstein's summation convention with respect to indices $i,j,k$ is assumed and  the argument $\omega$ is omitted  in the above equation for the simplicity of notation.

First, we introduce some deterministic function spaces related to our results. For $p >1$ and $\gamma \in \fR$, let
$H_{p}^{\gamma}=H_{p}^{\gamma}(\fR^{d})$ denote  the class of all
(tempered) distributions $u$  on $\fR^{d}$ such that
\begin{equation}
        \label{eqn norm}
\| u\| _{H_{p}^{\gamma}}:=\|(1-\Delta)^{\gamma/2}u\|_{L_{p}}<\infty,
\end{equation}
where
$$
(1-\Delta)^{\gamma/2} u = \cF^{-1} \left[(1+|\xi|^2)^{\gamma/2}\cF [u] \right].
$$
It is well-known that if $\gamma=1,2,\cdots$, then
$$
H^{\gamma}_p=W^{\gamma}_p:=\{u: D^{\alpha}_x u\in L_p(\fR^d), \, \,\,|\alpha|\leq \gamma\}, \quad \quad H^{-\gamma}_p=\left(H^{\gamma}_{p/{(p-1)}}\right)^*,
$$
where $\left(H^{\gamma}_{p/{(p-1)}}\right)^*$ is the dual space of $H^{\gamma}_{p/{(p-1)}}$.
For  a  tempered distribution $u\in H^{\gamma}_p$ and $\phi\in
\cS(\fR^d)$, the action of $u$ on $\phi$ (or the image of $\phi$
under $u$) is defined as
$$(u,\phi)=\left((1-\Delta)^{\gamma/2}u ,
(1-\Delta)^{-\gamma/2}\phi \right)=\int_{\fR^d}
(1-\Delta)^{\gamma/2}u(x) \cdot (1-\Delta)^{-\gamma/2}\phi(x) \,dx.
$$
Let  $l_2$ denote the set of all sequences $a=(a^1,a^2,\cdots)$ such that
$$|a|_{l_{2}}:= \left(\sum_{k=1}^{\infty}|a^{k}|^{2}\right)^{1/2}<\infty.
$$
By $H_{p}^{\gamma}(l_{2})=H_{p}^{\gamma}(\fR^d ; l_2)$  we denote the class of all $l_2$-valued
(tempered) distributions $v=(v^1,v^2,\cdots)$ on $\fR^{d}$ such that
$$
\|v\|_{H_{p}^{\gamma}(l_{2})}:=\||(1-\Delta)^{\gamma/2}v|_{l_2}\|_{L_{p}}<\infty.
$$
In particular, we set
$$
L_p:=H^0_p  \quad \text{and} \quad L_p(l_2):= H^0_p(l_2).
$$

To state our assumption for the initial data, we introduce the Besov space characterized by the Littlewood-Paley operator. See \cite[Chapter 6]{bergh1976interpolation} or 
 \cite[Chapter 6]{grafakos2009modern} for more details.  Let $\Psi$ be a 
  nonnegative function on $\fR^d$ so that $\hat \Psi \in C_c^\infty \left( B_{2}(0)  \setminus B_{1/2}(0)\right)$
and
\begin{align}
                    \label{sum 1}
 \quad \sum_{j \in \bZ} \hat \Psi (2^{-j} \xi) = 1, \quad \forall \xi \in \fR^d \setminus \{0\},
\end{align}
where $B_r(0) := \{ x \in \fR^d : |x| \leq r\}$ and $\hat \Psi$ is the Fourier transform of $\Psi$. 
For a tempered distribution $u$, we define
\begin{align}
				\label{del j}
\Delta_j u(x):=\Delta_j^\Psi u(x):= \cF^{-1} \left[ \hat \Psi (2^{-j} \xi) \cF u (\xi) \right] (x)
\end{align}
and
$$
S_0(u)(x) = \sum_{j=-\infty}^0 \Delta_j u (x),
$$
where the convergence is understood in the sense of distributions.
Due to \eqref{sum 1},
\begin{align}
                    \label{little iden}
u(x)= S_0(u)(x) + \sum_{j=1}^\infty \Delta_j u(x).
\end{align}
The Besov space $B^\gamma_p$ =$B^\gamma_p(\fR^d)$ with the order $\gamma$ and the exponent $p$  is  the space of all tempered distributions $u$ such that
\begin{equation}
    \label{Besov}
\|u\|_{B_p^\gamma}:=\| S_0(u) \|_{L_p} + \left(\sum_{j=1}^\infty 2^{\gamma pj} \| \Delta_j u \|_{L_p}^p \right)^{1/p} < \infty.
\end{equation}
Similarly,  the homogeneous Besov space $\dot{B}^\gamma_p$ =$\dot{B}^\gamma_p(\fR^d)$ with the order $\gamma$ and the exponent $p$  is  the space of all tempered distributions $u$ such that
\begin{equation*}
\|u\|_{\dot{B}_p^\gamma}:= \left(\sum_{j=-\infty}^\infty 2^{\gamma pj} \| \Delta_j u \|_{L_p}^p \right)^{1/p} < \infty.
\end{equation*}

\begin{remark} 
					\label{rem homo}
The followings  are well-known (cf. \cite{bergh1976interpolation,grafakos2009modern}).

(i)  Let $1<p<\infty$ and $\gamma>0$. Then, two norms $\|\cdot \|_{B_p^\gamma}$ and $\|\cdot \|_{\dot B_p^\gamma}+\|\cdot\|_{L_p}$ are equivalent, and  for any $c>0$,  we have
$\|u(cx)\|_{\dot B_p^\gamma}=c^{\gamma-d/p}\|u\|_{\dot B_p^\gamma}$.

(ii) Let $p\geq 2$. Then $B^{\gamma}_p \supset H^{\gamma}_p$ and $B^{\gamma}_p \subset H^{\gamma'}_p$ for any $\gamma'<\gamma$.

\end{remark}

Next, we introduce stochastic Banach spaces.  Denote
$$
\bB^{\gamma}_p :=L_p(\Omega, \rF_0; B^{\gamma}_p),
\quad \dot \bB^{\gamma}_p :=L_p(\Omega, \rF_0; \dot B^{\gamma}_p),
\quad \bH^{\gamma}_p :=L_p(\Omega, \rF_0; H^{\gamma}_p),
$$
and for a stopping time $\tau$ and  weight function $\delta=\delta(\omega,t) \geq 0$, 
denote
$$
\bH^{\gamma}_{p}(\tau,\delta)=L_p(\Omega\times [0,\tau], dP\times \delta(t)dt, \cP ; H^{\gamma}_p), \quad \bH^{\gamma}_{p}(\tau):=\bH^{\gamma}_{p}(\tau,1)$$
and
$$
\bH^{\gamma}_{p}(\tau,\delta,l_2)=L_p(\Omega\times [0,\tau], dP\times \delta(t)dt, \cP ; H^{\gamma}_p(l_2)), \quad 
\bH^{\gamma}_{p}(\tau,l_2):= \bH^{\gamma}_{p}(\tau,1,l_2).$$
For the notational convenience, we use $\bL_p$  instead of $\bH^0_{p}$.
We write 
$$
u \in L_p \left( \Omega, \rF; C\left( [0,\tau]; H^\gamma_p \right) \right)
$$ 
if $u$ is an $H^\gamma_p$-valued predictable process such that $u (\omega)\in C\left( [0,\tau(\omega)] ;H_p^\gamma \right)$ (a.s.),
and
$$
\bE  \sup_{t \leq \tau} \|u\|^p_{H_p^\gamma} <\infty.
$$
\begin{remark}
It is easy to check that  if $\tau$ is bounded, then  $L_p \left( \Omega, \rF ; C\left( [0,\tau] ; H^\gamma_p \right) \right)$ is a Banach space. 
\end{remark}

Let $\cD$ be the space of distributions (generalized functions) on $C_c^\infty(\fR^d)$, 
and let $\cD(l_2)$ denote the space of  $l_2$-valued distributions (generalized functions) on $C_c^\infty(\fR^d)$.

\begin{definition}
					\label{def sol-2}
Let $u_0$ be $\cD$-valued random variable, $u$ and $f$ be $\cD$-valued predictable stochastic processes, and $g$ be $\cD(l_2)$-valued predictable stochastic process. 
We say that $u$ satisfies (or is a solution to) the equation
\begin{align}
						\notag
&du(t,x)=f(t,x) dt +g(t,x) dw^k_t, \quad (t,x) \in [0,\tau]  \times \fR^d \\
						 \label{eqn dist}
&u(0,\cdot)=u_0
\end{align}
in the sense of distributions if   for any $\phi\in C^{\infty}_0(\fR^d)$, the  equality
\begin{align}
					\label{def eq 2}
(u(t,\cdot),\phi)=   (u_0 ,\phi) + \int^t_0 (f(s,\cdot),\phi)ds+\sum_k \int^t_0 (g^k(s,\cdot),\phi)dw^k_t
\end{align}
holds for all $t \leq \tau$ (a.s.).  

In particular, if $u_0\in \bL_p$, $u,f\in \bL_p(\tau)$, and $g \in \bL_p(\tau,l_2)$, then  $u$ is a  solution to \eqref{main eqn} if  
\begin{eqnarray*}
						\label{def sol}
(u(t,\cdot),\phi)
&=&   (u_0 ,\phi) + \int^t_0  \left[ \left(  a^{ij}(t)u(s,\cdot), \phi_{x^ix^j} \right) + (f(u(s,\cdot),s,\cdot),\phi) \right]ds \\
&& +\sum_k \int^t_0 \left[ - \left(\sigma^{ik}(t)u(s,\cdot), \phi_{x^i} \right) + (g^k(u(s,\cdot),s,\cdot),\phi)   \right]dw^k_t
\end{eqnarray*}
for all $t \in [0,\tau]$ $(a.s.)$.

\end{definition}

\begin{remark}
					\label{moli rem}
Suppose that  $u_0\in \bL_p$, $u,f\in \bL_p(\tau)$, and $g \in \bL_p(\tau,l_2)$. In Definition \ref{def sol-2}, the subset $\Omega' \subset \Omega$ such that $P(\Omega')=1$  and \eqref{def eq 2} holds for all $(\omega', t) \in \Omega' \times [0,\tau(\omega')]$
depends  on the test function $\phi$.
However, taking the countable dense subset of $C_c^\infty(\fR^d)$ in $L_q(\fR^d)$ with $q= \frac{p}{p-1}$, one can find a $\Omega' \subset \Omega$ such that 
$P(\Omega')=1$  and \eqref{def eq 2} holds for all $(\omega', t) \in \Omega' \times [0,\tau(\omega')]$ and $\phi \in C_c^\infty(\fR^d)$. 

Due to the above fact, one can use Sobolev's mollifier to approximate $u$ with smooth functions as in the deterministic case. Indeed, 
let $\phi \in C_c^\infty(\fR^d)$ have a unit integral,  and denote
$\phi^{\varepsilon}(x) = \frac{1}{\varepsilon^{d}} \phi(x/\varepsilon)$. 
Plugging in $\phi^\varepsilon(x- \cdot)$ in \eqref{def eq 2} in place of $\phi$, we get
\begin{align*}
u^\varepsilon(t,x)=u^{\varepsilon}_0+\int^t_0f^\varepsilon(s,x)ds+\sum_k \int^t_0g^{k,\varepsilon}(s,x)dw^k_t
\end{align*}
for all $t \leq \tau$, $x \in \fR^d$,  $(a.s.)$, where
$$
u^{\varepsilon}(t,x) = \int_{\fR^d} u(t,y) \phi^\varepsilon (x-y) dy, \quad u^{\varepsilon}_0(x)=\int_{\fR^d} u_0(y) \phi^\varepsilon (x-y) dy
$$
and
$$
f^{\varepsilon}(t,x) = \int_{\fR^d} f(t,y) \phi^\varepsilon (x-y) dy, \quad g^{k,\varepsilon}(t,x) = \int_{\fR^d} g^k(t,y) \phi^\varepsilon (x-y) dy.
$$
\smallskip

\end{remark}

Now we introduce our assumptions on  the coefficients $a^{ij}(t)$ and $\sigma^{ik}(t)$. 
Set
$$
\alpha^{ij}(t) :=a^{ij}(t)-\frac{1}{2}(\sigma^i(t),\sigma^j(t))_{l_2} 
$$
and
$$
|\sigma(t)|=\max_{i=1,\cdots,d}|\sigma^i(t)|_{l_2}.
$$
\begin{assumption}
					\label{co as}
\begin{enumerate}[(i)]
\item The coefficients $a^{ij}(t)$, $\sigma^{ik}(t)$  are predictable  for all $i$, $j$, $k$, and
 $$
 \alpha^{ij}(t) \xi^i \xi^j \geq 0, \quad \forall  (\omega, t,\xi) \in  \Omega \times (0,\infty) \times \fR^d.
 $$

\item The coefficients $a^{ij}(t)$, $|\sigma^{ik}(t)|^2$  are  locally  integrable, i.e.
\begin{align}
						\label{2019011510}
\int_0^t  |a^{ij}(s)| ds 
+\int_0^t  \sum_{k=1}^\infty \left|\sigma^{ik}(s) \right|^2 ds < \infty ~(a.s.) \qquad \forall t>0, i,j
\end{align}
\end{enumerate}
\end{assumption}

\begin{remark}
(i)  Obviously, Assumption \ref{co as} allows the coefficients to be unbounded or degenerate.
\smallskip

(ii) Without loss of generality, we 
 may assume that the coefficients $a^{ij}(t)$ and $\alpha^{ij}(t)$ are symmetric, i.e.
$$
a^{ij}(t)=a^{ji}(t) \quad \text{and} \quad \alpha^{ij}(t) = \alpha^{ji}(t) \qquad \forall i,j.
$$
Thus if we denote  by $\delta(t)$ the smallest eigenvalue of the matrix $(\alpha^{ij}(t))$, then Assumption \ref{co as} implies 
 \begin{align}
					\label{2019032701}
 \alpha^{ij}(t) \xi^i \xi^j \geq  \delta(t) |\xi|^2 \geq 0 \quad \forall  (\omega, t,\xi) \in  \Omega \times (0,\infty) \times \fR^d.
 \end{align}

\end{remark}

Here is the main result of this article.
\begin{theorem}
							\label{main thm}
Let $p \in [2, \infty)$, $T \in [0,\infty)$, $\delta(t)$ be the smallest eigenvalue of $\alpha^{ij}(t)$,
$\tau\leq T$ be a stopping time, $\gamma \in \fR$, $u_0 \in \bB_p^{\gamma+2(1-1/p)}$, $f \in \bH^\gamma_p \left(\tau\right) \cap \bH^\gamma_p \left(\tau,\delta^{1-p}\right)$, 
and $g \in \bH^{\gamma }_p \left(\tau, l_2\right) \cap \bH^{\gamma+1}_p( \tau, \delta^{1-p/2},l_2)$.
Suppose that  Assumption \ref{co as} holds and 
$$
g_x \in \bH_p^{\gamma}( \tau, |  \sigma|^p ,l_2) \cap \bH_p^{\gamma}( \tau, |\sigma|^p \delta^{1-p},l_2).
$$
Then    there exists a unique solution 
$u \in L_p \left( \Omega, \rF; C\left( [0,\tau] ; H^\gamma_p \right) \right) $ to \eqref{main eqn}, and for this solution we have
\begin{align}
						\label{2019011031}
\bE \sup_{t  \in [0,\tau] } \|u(t,\cdot)\|^p_{H^\gamma_p}
\leq N_1 \left(
\|u_0\|^p_{\bH^\gamma_p} 
+\|f\|^p_{\bH^\gamma_p(\tau)} 
+\|g\|^p_{\bH^\gamma_p(\tau)} +\|g_x\|^p_{\bH^\gamma_p( \tau, |\sigma|^p,l_2)} \right),
\end{align}
and
\begin{eqnarray}
							\notag
\|u_{xx}\|_{\bH^\gamma_p(\tau,\delta)}   &\leq& N_2 \Big(
									\label{2019011320}
 \|u_0\|_{\bB_p^{\gamma+2 \left(1-1/ p \right)}}  +  \|  f\|_{\bH^\gamma_p( \tau,\delta^{1-p} )} \\
 && \quad +\|g_x\|_{\bH^\gamma_p( \tau, |\sigma|^p \delta^{1-p},l_2)}+ \|  g_x\|_{\bH^\gamma_p( \tau,\delta^{1-p/2},l_2)} \Big),
\end{eqnarray}
where  $N_1= N_1(p,T)$ and $N_2=N_2(d,p)$.
\end{theorem}
The proof of this theorem is given in Section \ref{pf main thm}.

\begin{remark}
						\label{ap rem}
(i) A nonnegative function $w(x)$ is said to be of class $A_p$ if
\begin{align*}
\sup \left( \frac{1}{|Q|} \int_{Q} w(x) dx \right) \left( \frac{1}{|Q|} \int_{Q} w(x)^{-1/(p-1)} dx \right)^{p-1} <\infty,
\end{align*}
where the sup is taken over all cubes on $\fR^{d}$ (cf. \cite[Section 7.1]{grafakos2014classical}).
Note that, due to the term $\frac{1}{|Q|} \int_{Q} w(x)^{-1/(p-1)} dx$, it is required that
$$
w(x)  > 0 \qquad (a.e.).
$$
However since our coefficients $a^{ij}(t)$ and $\sigma^{ik}(t)$ can be degenerate on sets with positive measures, our weights  are generally not  in $A_p$-class, which makes us unable to use  $A_p$-weight theories from the Fourier analysis.
\smallskip

(ii)    Suppose that $a^{ij}$ and $|\sigma^{i}|_{l_2}$ are bounded.   Then, since $\delta$ is also bounded,  the conditions for $f$ and $g$ in Theorem \ref{main thm} are as follows: 
$$
f\in \bH^\gamma_p \left(\tau\right) \cap \bH^\gamma_p \left(\tau,\delta^{1-p}\right) = \bH^\gamma_p \left(\tau, \delta^{1-p}\right) ,
 $$
$$
g\in  \bH^{\gamma}_p \left(\tau\right) \cap \bH^{\gamma+1}_p( \tau, \delta^{1-p/2},l_2) = \bH^{\gamma}_p \left(\tau, \delta^{1-p/2},l_2\right) ,
 $$
and
$$
g_x \in \bH_p^{\gamma}( \tau, |\sigma|^p ,l_2) \cap \bH_p^{\gamma}(  \tau,  |\sigma|^p \delta^{1-p} ,l_2)    \supset  \bH_p^{\gamma}( \tau,l_2)  \cap \bH_p^{\gamma}( \tau,\delta^{1-p} ,l_2) =  \bH_p^{\gamma}( \tau,\delta^{1-p},l_2)   .
$$
(iii) If  the matrix $(\alpha^{ij}(t))$  is uniformly elliptic,  that is,  there exists a positive constant $\varepsilon>0$ such that $\delta(t)\geq \varepsilon$, then   in Theorem \ref{main thm} it is assumed that 
$$
 f \in\bH^\gamma_p \left(\tau\right) \cap \bH^\gamma_p \left(\tau,\delta^{1-p}\right) = \bH^\gamma_p \left(\tau\right) ,
 $$
$$
 g \in \bH^{\gamma}_p \left(\tau , l_2\right) \cap \bH^{\gamma+1}_p( \tau, \delta^{1-p/2},l_2) = \bH^{\gamma}_p \left(\tau , l_2 \right) ,
 $$
and
$$
g_x \in \bH_p^{\gamma}( \tau, |\sigma|^p, l_2) \cap \bH_p^{\gamma}(\tau, |\sigma|^p \delta^{1-p},l_2)   = \bH_p^{\gamma}( \tau, |\sigma|^p ,l_2) .
$$
Furthermore, if $\delta(t)\geq \varepsilon$ and $|\sigma|_{l_2}$ is bounded,   then our data spaces are given by
$$
 f \in\bH^\gamma_p \left(\tau\right) \cap \bH^\gamma_p \left(\tau,\delta^{1-p}\right) = \bH^\gamma_p \left(\tau\right) ,
 $$
$$
 g \in \bH^{\gamma}_p \left(\tau, l_2\right) \cap \bH^{\gamma}_p( \tau, \delta^{1-p/2},l_2) = \bH^{\gamma}_p \left(\tau\right) ,
 $$
 $$
g_x   \in \bH_p^{\gamma}( \tau, |\sigma|^p ,l_2) \cap \bH_p^{\gamma}(\tau, |\sigma|^p \delta^{1-p} ,l_2)   \supset \bH_p^{\gamma}(\tau, \delta^{1-p} ,l_2)  =  \bH_p^{\gamma}( \tau,l_2) .
$$
Therefore our data spaces for  $f$, $g$ obviously include the classical data spaces (cf. \cite{Krylov1999}). 
\smallskip

(iv) If $p=2$, then $1-p/2=0$. Thus $\delta(t)^{1-p/2}$ is not well-defined if  $\delta(t)=0$. 
In this case, we define $\delta(t)^{1-p/2}=1$. 

\smallskip

(v) We chose the smallest eigenvalue $\delta(t)$ of $(a^{ij}(t))$ as the weight in our results. 
However, it is possible that Theorem \ref{main thm} holds with  any function $\delta(t)$ satisfying \eqref{2019032701} in place of the smallest eigenvalues. 
\end{remark}

\mysection{Deterministic  linear equations}

In this section, we consider the following deterministic  equation on $\fR^d$:
\begin{align}				
					\label{2018122201}
 \frac{du}{dt} =  a^{ij}(t)u_{x^ix^j}+f, \quad t\in (0,T]\,; \quad u(0,\cdot)=u_0.
\end{align}
The coefficients $a^{ij}$ depend only on $t$. We say that $u$ is a (weak) solution to \eqref{2018122201}  if  \eqref{2018122201} holds in the sense of distributions, that is, for any $\phi\in C^{\infty}_c(\fR^d)$ the equality 
  \begin{align}
							\notag
&\int_{\fR^d}u(t,x)\phi(x) dx  \\
								\label{2018011010}
&= \int_{\fR^d} u_0(x) \phi(x)dx +  \int_0^t \int_{\fR^d} \left(  a^{ij}(s)u(s,x) \phi_{x^ix^j}(x)+f(t,x) \phi(x)  \right) dx ds
\end{align}
holds for all $t\leq T$.

Here we assume
 \begin{align}
							\label{2018122202}
a^{ij}(t) \xi^i \xi^j  \geq 0,\quad \forall  ( t,\xi) \in  (0,\infty) \times \fR^d 
 \end{align}
and set
$$
|a(t)| = \max_{i,j} |a^{ij}(t)|.
$$
We emphasize that there is no bounded assumption on $a^{ij}(t)$.
However,  to make sense of   equality \eqref{2018011010},   it is at least required that
\begin{align*}
\left| \int_0^t \int_{\fR^d}   a^{ij}(s)u(s,x) \phi_{x^ix^j}(x) dxds\right| < \infty \qquad \forall t \in [0,T],
\end{align*}
which holds if 
$$
u \in L_1\left( (0,T), |a(t)| dt ; L_p   \right), \quad p>1.
$$
Indeed, if  $u \in L_1\left( (0,T), |a(t)| dt ; L_p   \right)$, then by H\"older's inequality, with $q= \frac{p}{p-
1}$,
\begin{align}
							\label{2019011410}
\left| \int_0^t \int_{\fR^d}   a^{ij}(s)u(s,x) \phi_{x^ix^j}(x) dxds\right|
\leq \|\phi_{xx}\|_{L_q}  \int_0^t  \|u(s,\cdot)\|_{L_p}   |a(s)|ds  < \infty.
\end{align}

Moreover, if
 \begin{align}
						\label{2018122501}
\int_0^T |a(t)| dt < \infty
\end{align}
and $\sup_{t \in [0,T]} \|u(t,\cdot)\|_{L_p(\fR^d)} <\infty$, then
\begin{align}
						\label{2019011701}
 \int_0^t  |a(s)| \|u(s,\cdot)\|_{L_p} ds  
 \leq \sup_{s \leq T} \|u(s,\cdot)\|_{L_p} \int_0^t |a(s)| ds < \infty.
\end{align}

\begin{lemma}[A priori estimate]
					\label{deter a priori}
Let $p \in (1,\infty)$, $T \in (0,\infty)$, 
$f \in L_p((0,T) ; L_p )$, $u_0 \in L_p$, and  \eqref{2018122202} holds. Suppose that $u$ is a solution  to equation \eqref{2018122201}
and
$$
u \in C\left( [0,T] ; L_p \right) \cap L_1\left( (0,T), |a(t)|dt ; L_p   \right).
$$
Then  
\begin{align}
						\label{well 1}
\sup_{t \in [0,T]} \|u(t,\cdot)\|^p_{L_p} 
\leq N\left(  \|u_0\|^p_{L_p}  +  \int_0^T \|f(t,\cdot)\|^p_{L_p} dt  \right),
\end{align}
where $N=N(p,T)$.
\end{lemma}

\begin{proof}
If  the coefficients are bounded, then the lemma is a classical result and  can be found, for instance, in \cite{krylov1986characteristics,gerencser2015solvability}.
The proof for general case is similar. Nonetheless, we give a detailed proof  for the sake of the completeness.  

We use Sobolev mollifiers. Fix a nonnegative $\phi \in C_c^\infty(\fR^d)$ such that $0\leq \phi \leq 1$, $\int_{\fR^d}\phi \,dx=1$, and $\phi(x)=1$ near $x=0$.  Denote
$\phi^\varepsilon(x)= \varepsilon^{-d} \phi(x/\varepsilon)$, $u_0^\varepsilon(x)= u_0 \ast \phi^\varepsilon(x)$, and $u^\varepsilon (t,x) = u(t,\cdot) \ast \phi^\varepsilon(\cdot)(x)$.
Putting $\phi^\varepsilon(x-\cdot)$ in \eqref{2018011010}, for all $(t,x) \in (0,T) \times \fR^d$, we have
\begin{align}
\label{eqn 3.18.1}
u^\varepsilon(t,x) = u_0^\varepsilon(x) + \int_0^t a^{ij}(s)u^{\varepsilon}_{x^ix^j}(s,x)ds +\int_0^t f^\varepsilon(s,x)ds.
\end{align}
Note that \eqref{2018011010} and \eqref{eqn 3.18.1} make sense due to \eqref{2019011410}.  By the chain rule, for any $p>1$,
$$
\frac{d}{dt} (|u^\varepsilon|^p) =  p|u^\varepsilon|^{p-2}u^\varepsilon u^\varepsilon_t, \quad (0^{p-2}\times 0:=0)
$$
and thus by the Fundamental theorem of calculus 
\begin{align}
						\notag
|u^\varepsilon(t,x)|^p 
&=|u^\varepsilon_0(x)|^p+\int_0^t p|u^\varepsilon|^{p-2}(s,x) u^\varepsilon(s,x) a^{ij}(s)u^\varepsilon_{x^i x^j}(s,x)ds \\
						\label{2019012301}
&\qquad + \int_0^t p |u^\varepsilon(s,x)|^{p-2}u^\varepsilon(s,x) f^\varepsilon(s,x) ds.
\end{align}

To apply Fubini's theorem we first note that, since $\|u^{\varepsilon}(t)\|_p+\|u^{\varepsilon}_{xx}(t)\|_p\leq N(\varepsilon)\|u(t)\|_{L_p}$, by H\"older's inequality
\begin{eqnarray*}
\int^t_0 \int_{\fR^d} |u^{\varepsilon}|^{p-1}|a^{ij}| |u^{\varepsilon}_{x^ix^j}| dxds &\leq& \int^t_0 \|u^{\varepsilon}(s)\|^{p-1}_p\|u^{\varepsilon}_{xx}(s)\|_{p}|a(s)|ds\\
&\leq&N \sup_{r\leq T} \|u(r)\|^{p-1}_p \int^t_0 \|u(s)\|_p |a(s)|ds<\infty.
\end{eqnarray*}
Thus, integrating both sides of \eqref{2019012301} with respect to $x$, and applying   Fubini's theorem and the integration by parts, we have
\begin{align*}
&\int_{\fR^d}|u^\varepsilon(t,x)|^p  dx \\
&= \int_{\fR^d}|u^\varepsilon_0(x)|^pdx
+\int_0^t \int_{\fR^d} p|u^\varepsilon|^{p-2}(s,x) u^\varepsilon(s,x) a^{ij}(s)u^\varepsilon_{x^i x^j}(s,x)ds dx  \\
&\qquad + \int_0^t  \int_{\fR^d}p |u^\varepsilon(s,x)|^{p-2}u^\varepsilon(s,x) f^\varepsilon(s,x) dx ds\\
&= \int_{\fR^d}|u^\varepsilon_0(x)|^pdx
-\int_0^t \int_{\fR^d} p(p-1)|u^\varepsilon|^{p-2}(s,x) u_{x^j}^\varepsilon(s,x)a^{ij}(s)u^\varepsilon_{x^i}(s,x)ds dx \\
&\quad + \int_0^t \int_{\fR^d} p|u^\varepsilon|^{p-2}(s,x) u^\varepsilon(s,x) f^\varepsilon(s,x)dxds.
\end{align*}
Due to \eqref{2018122202},
\begin{align*}
&\int_0^t \int_{\fR^d} p(p-1)|u^\varepsilon|^{p-2}(s,x) u_{x^j}^\varepsilon(s,x)a^{ij}(s)u^\varepsilon_{x^i}(s,x)ds dx \geq 0.
\end{align*}
Thus
\begin{align*}
&\sup_{t \in [0,T]} \int_{\fR^d}|u^\varepsilon(t,x)|^p  dx \\
&\leq \int_{\fR^d}|u^\varepsilon_0(x)|^pdx
+ \int_0^T \int_{\fR^d} p|u^\varepsilon|^{p-2}(s,x)  \left| u^\varepsilon(s,x) \right|  \left| f^\varepsilon(s,x) \right| dxds.
\end{align*}
By H\"older's  inequality and Young's inequality,   for any constant $c>0$
\begin{align*}
\int_0^T \int_{\fR^d} |u^\varepsilon|^{p-2}(s,x) u^\varepsilon(s,x) f^\varepsilon(s,x)dxds
\leq \int_0^T \|  c u^{p-1}(s,\cdot)\|_{L_q(\fR^d)}  \| c^{-1}f\|_{L_p(\fR^d)} ds \\
\leq \frac{1}{q} c^q \int_0^T    \| u (s,\cdot)\|^p_{L_p(\fR^d)}  ds +   c^{-p} \frac{1}{p} \int_0^T  \| f\|^p_{L_p(\fR^d)} ds \\
\leq \frac{1}{q} c^q T  \sup_{s \leq T} \| u (s,\cdot)\|^p_{L_p(\fR^d)}  +   c^{-p} \frac{1}{p} \int_0^T  \| f\|^p_{L_p(\fR^d)} ds ,
\end{align*}
where $q= \frac{p}{p-1}$.
Therefore taking $c>0$ small so that $\frac{p}{q} c^q T <1$, 
we obtain 
\begin{align*}
\sup_{t \in [0,T]} \int_{\fR^d}|u^\varepsilon(t,x)|^p  dx
\leq N\left(\int_{\fR^d}|u^\varepsilon_0(x)|^pdx
+ \int_0^T \int_{\fR^d} |f^\varepsilon(s,x)|^p dxds \right),
\end{align*}
where $N$ depends only on $p$ and $T$. 
Observing
\begin{align*}
&\left(\int_{\fR^d}|u^\varepsilon_0(x)|^pdx+ \int_0^T \int_{\fR^d} |f^\varepsilon(s,x)|^p dxds \right) \\
&\leq \left(\int_{\fR^d}|u_0(x)|^pdx+ \int_0^T \int_{\fR^d} |f(s,x)|^p dxds \right),
\end{align*}
and using $u^{\varepsilon}\to u$ in $C(  [0,T];L_p)$,  we finally get (\ref{well 1}).
\end{proof}

\begin{remark}
If  \eqref{2018122501} holds, then by \eqref{2019011410} and \eqref{2019011701}, 
$$
C\left( [0,T] ; L_p \right) \cap L_1\left( (0,T), |a(t)|dt ; L_p   \right)
=C\left( [0,T] ; L_p \right).
$$ 
\end{remark} 

In Lemma \ref{deter a priori},  local  integrability  of the coefficients $a^{ij}(t)$  is not assumed.
However,   \eqref{2018122501} is needed for the proof of  the existence  as follows.

\begin{theorem}[Well-posedness]
					\label{well thm}
Let $p \in (1,\infty)$, $T \in (0,\infty)$, $f \in L_p((0,T) ; L_p ) $, and $u_0 \in L_p$.
 Suppose that \eqref{2018122202} and \eqref{2018122501} hold.
Then   there exists a unique solution $u \in C\left( [0,T] ; L_p \right) $ to equation \eqref{2018122201} such that
\begin{align}
						\label{well 1-2}
\sup_{t \in [0,T]} \|u(t,\cdot)\|^p_{L_p} 
\leq N\left( \|u_0\|^p_{L_p} + \int_0^T \|f(t,\cdot)\|^p_{L_p}\,dt \right),
\end{align}
where $N$ depends only on $p$ and $T$.
\end{theorem}
\begin{proof}

We remark that the theorem is a classical result if the coefficients are bounded, and we give a proof for the general case for the sake of the completeness.

{\bf Part I.} (Estimate and Uniqueness)
Due to \eqref{2018122501},
$$
C\left( [0,T] ; L_p \right) \cap L_1\left( (0,T), |a(t)|dt ; L_p   \right)=C\left( [0,T] ; L_p \right) .
$$
By this and  Lemma \ref{deter a priori}, (\ref{well 1-2}) holds if  $u \in C\left( [0,T] ; L_p \right)$ is a solution to equation (\ref{2018122201}), and the uniqueness also follows. 
\smallskip

\bigskip

{\bf Part II.} (Existence)
\smallskip

Let $W'_t=(W'^1_t,\cdots, W'^d_t)$ be a $d$-dimensional Wiener process on a probability space $(\Omega', \rF',P')$. 
Since  $A(t):=(a^{ij}(t))$ is a nonnegative symmetric matrix, there exists a nonnegative symmetric (non-ramdon) matrix $\sigma'(t)=(\sigma'_{ij}(t))$ such that
$$
2A(t)= (\sigma')^2(t).
$$
Due to \eqref{2018122501}, $\sigma'(t)$ is It\^o integrable (cf. \cite[Chapter 6.3]{Krylov2002}),  i.e.
$$
\int_0^t  |\sigma'(s)|^2 1_{s \leq T} ds <\infty, \quad \forall t.
$$
We define 
\begin{align}
					\label{2018122211}
X'_t := \int_0^t \sigma'(t)dW'_t,  \quad (i.e.~~ X'^i_t=\sum_{k=1}^d \int^t_0 \sigma'_{ik}(s)dW'^k_s, ~ (i=1,2,\cdots,d)).
\end{align}
We will first show that $u(t,x)$ defined as
\begin{align}
						\label{2018122203}
u(t,x) :=  \bE'[u_0(x+X'_t)] + \int_0^t\bE'[f(s,x+X'_t-X'_s )] ds
\end{align}
is a solution to equation (\ref{2018122201}) if $u_0$ and $f$ are sufficiently smooth,
where $\bE'$ is the expectation in the probability space $(\Omega', \rF',P')$.
Then by using an approximation, we finally  prove the existence of a solution for general $u_0$ and $f$.

We divide the details   into several  steps.

\begin{enumerate}[(i)]
\item  Let $u_0 \in C^2 \cap H^2_p$ and $f=0$. Then by It\^o's formula, for all $t \in [0,T]$,
\begin{align}
\nonumber
u(t,x):= \bE'[u_0(x+X'_t)] 
&= \bE'[u_0(x)] + \int_0^t a^{ij}(s)  \bE'\left[\frac{\partial^2 u_0}{\partial x^i \partial x^j}(x+X'_s) \right]ds \\
&= u_0(x) + \int_0^t a^{ij}(s)  u_{x^ix^j}(s,x)ds.   \label{eqn 3.19.1}
\end{align}
Thus $u(t,x)$ satisfies equation \eqref{2018122201}.   Also note that  
\begin{eqnarray*}
\int^t_0 \|a^{ij}(s) u_{x^ix^j}\|_p \,ds&\leq& \int^t_0 |a(s)| \bE' \left\|\frac{\partial^2 u_0}{\partial x^i \partial x^j}(x-X'_s) \right\|_p \,ds\\
&\leq& \|u_0\|_{H^2_p}\int^T_0 |a(s)|ds<\infty.
\end{eqnarray*}
Therefore, from  \eqref{eqn 3.19.1} it easily follows that   $u \in  C\left( [0,T] ; L_p \right)$ .

\item  Let $u_0 =0$ and $f \in L_1\left((0,T);C^2 \cap H^2_p \right)$. 
 Applying  (a generalized) It\^o's formula (see e.g. Theorem 4.1.1 or  Corollary 4.1.2 in \cite {Krylov1995}),  we get for each $t>s$, 
\begin{align}
						\label{2019011901}
\bE'[f(s,x+X'_t-X'_s)]
=  f(s,x) + \int_s^t a^{ij}(r)  \bE'\left[f_{x^ix^j}(s,x+X'_r-X'_s) \right]dr.
\end{align}
By integrating the above terms with respect to $s$ from $0$ to $t$ and the Fubini theorem, 
\begin{align}
\nonumber
u(t,x)
&:=\int_0^t\bE'[f(s,x+X'_t-X'_s)]ds \\
&=\int_0^tf(s,x)ds +  \int_0^t \int_s^t a^{ij}(r)  \bE'\left[f_{x^ix^j}(s,x+X'_r-X'_s) \right]dr ds  \nonumber\\
&=\int_0^tf(s,x)ds +  \int_0^t a^{ij}(r) \int_0^r   \bE'\left[f_{x^ix^j}(s,x+X'_r-X'_s) \right] ds dr   \label{eqn 3.20.1}\\
&=\int_0^tf(s,x)ds +\int_0^t a^{ij}(r) u_{x^ix^j}(r,x) dr. \nonumber
\end{align}
Therefore $u(t,x)$ is a solution to equation \eqref{2018122201}.
The inclusion $u \in  C\left( [0,T] ; L_p \right) $ can be easily obtained from \eqref{eqn 3.20.1} as was shown in (i) if $f \in L_1\left((0,T);C^2 \cap H^2_p\right)$.

\item (General Case)
Choose sequences 
$$u^n_0 \in C_c^\infty(\fR^d), \quad 
f^n \in  C\left([0,T];C^2 \cap H^2_p\right),
$$
so that as $n \to \infty$,
$$
u^n_0 \to u_0 \quad \text{in} \quad L_p \qquad \text{and} \qquad f^n \to f \qquad \text{in} \quad L_p((0,T) ; L_p ).
$$
Then by (i) and (ii), for all $n \in \bN$
\begin{align}
						\label{ge sol re}
u^n(t,x) := \bE'[u^n_0(x+X'_t)] + \int_0^t\bE'[f^n(s,x+X'_t-X'_s) ] ds
\end{align}
satisfies
 \begin{align*}
&u^n_t(t,x)=a^{ij}(t)u^n_{x^ix^j}(t,x)+f^n(t,x) \qquad (t,x) \in (0,T] \times \fR^d \\
&u^n(0,x)=u^n_0(x).
 \end{align*}
\end{enumerate}
Moreover, due to \eqref{well 1}, for all $n, m \in \bN$
\begin{align*}
\sup_{t \in [0,T]} \|(u_n-u_m)(t,\cdot)\|^p_{L_p} 
\leq N\left( \int_0^T \| (f_n - f_m)(t,\cdot)\|^p_{L_p} dt + \|u^n_0-u^m_0\|^p_{L_p} \right).
\end{align*}
Thus $u^n$ becomes a Cauchy sequence in $C\left( [0,T] ; L_p \right)$ 
and thus there exists a $u \in C\left( [0,T] ; L_p \right)$ such that
$u_n \to u$ in $C\left( [0,T] ; L_p \right)$ as $n \to \infty$.   Also, using \eqref{2018011010}  corresponding to $(u_n, f_n,u^n_0)$, and then taking $n\to \infty$, we easily find that  $u$ is a solution to equation \eqref{2018122201}. 
The theorem is proved.
\end{proof}

\begin{remark}
					\label{no smooth}
\begin{enumerate}[(i)]

\item Due to the approximation used in the proof of Theorem \ref{well thm}, for general  $u_0\in L_p$ and  $ f\in L_p((0,T) ; L_p )$, the solution $u$ to \eqref{2018122201} is given by
\begin{align}
					\label{0503 e 1}
 u(t,x)= \bE'[u_0(x+X'_t)] + \int_0^t\bE'[f(s,x+X'_t-X'_s )] ds.
\end{align}

More generally, following the proof of the theorem, one can check that  $u$ defined in \eqref{0503 e 1}  belongs to $C([0,T]; L_p)$ and becomes a solution to \eqref{2018122201} under a weaker condition, that is,  if $u_0 \in L_p$ and $f \in  L_1((0,T) ; L_p )$.  

Indeed, in the above approximation, we can  take 
 $u^n_0 \in C^2\cap H^2_p$ and $f^n \in L_1((0,T); C^2 \cap H^2_p)$ so that, as $n \to \infty$,
$$
u_0^n \to u_0 \quad \text{in}  \quad L_p,  \qquad f^n \to f \quad \text{in} \quad L_1((0,T);L_p).
$$
Take $u^n$ from \eqref{ge sol re}, then by Minkowski's inequality and the translation invariant of the $L_p$-norm,
\begin{align}
								\notag
&\| u^n(t,\cdot) - u(t,\cdot) \|_{L_p}  \\
								\notag
&\leq\left\|\bE'[(u^n_0-u_0)(\cdot + X'_t)]  \right\|_{L_p} +  \left \|\int_0^t\bE'[(f^n -f)(s,\cdot +X'_t-X'_s )] ds \right\|_{L_p}  \\
								\notag
&\leq   \bE' \left[  \left\| (u^n_0-u_0)(\cdot  + X'_t)  \right\|_{L_p} \right]  + \int_0^t\bE' \left[ \left\|(f^n -f)(s,\cdot+X'_t-X'_s ) \right\|_{L_p}  \right] ds \\
								\label{2019032120}
&= \left\| (u^n_0-u_0)(\cdot)  \right\|_{L_p}+ \int_0^t \left\|(f^n -f)(s,\cdot) \right\|_{L_p}  ds .
\end{align}
Also, by \eqref{2019032120},
\begin{align}
								\notag
& \int_0^T |a(t)|\| u^n(t,\cdot) - u(t,\cdot) \|_{L_p}   dt \\
								\label{2019032121}
&\leq \int_0^T|a(t)|dt  \left\| (u^n_0-u_0)(\cdot)  \right\|_{L_p}+  \int_0^T|a(t) |dt \int_0^T \left\|(f^n -f)(s,\cdot) \right\|_{L_p}  ds .
\end{align}
Therefore for any $\phi \in C_c^\infty(\fR^d)$ and $t \in (0,T)$, taking $n\to \infty$ to the equality
\begin{align*}
\left( u^n(t,\cdot), \phi \right)
= \left( u^n_0, \phi \right) + \int_0^t a^{ij}(s) \left( u^n(s,\cdot), \phi_{x^ix^j} \right) ds + \int_0^t \left(f^n(s,\cdot), \phi \right) ds,
\end{align*}
we get
\begin{align*}
\left( u(t,\cdot), \phi \right)
= \left( u_0, \phi \right) + \int_0^t a^{ij}(s) \left( u(s,\cdot), \phi_{x^ix^j} \right) ds + \int_0^t \left(f(s,\cdot), \phi \right) ds.
\end{align*}
In other words, the function $u$ defined in \eqref{0503 e 1} is a solution to \eqref{2018122201} if  $u_0 \in L_p$ and $f \in  L_1((0,T) ; L_p )$.
Moreover,  by \eqref{2019032120},   $\sup_{t\leq T} \|u^n(t)-u(t)\|_p \to 0$ as $n\to \infty$, and therefore
$$
u \in  C\left( [0,T] ; L_p \right) .
$$

\item
Let $ h \in C_0^2(\fR^d)$. 
Recall
$$
X'_t -X'_r = \int_r^t \sigma'(s)dW'_s.
$$
Note that, since $\sigma'$ is not random, both  $X'_t -X'_r$ and 
$\int_0^{t-r} \sigma'(t-s)dW'_s$  have Gaussian distributions with mean zero and the same covariance, and therefore they have the same distribution.
Thus by It\^o's formula and a change of variables,
\begin{align*}
\bE'[h(x+X'_t-X'_r)]
&=\bE'\left[h\left(x+ \int_0^{t-r} \sigma'(t-s)dW'_s \right) \right] \\
&=  h(x)  +\int_0^{t-r} a^{ij}(t-s)  \bE'\left[h_{x^ix^j} \left(x+\int_0^{s} \sigma'(t-\rho)dW'_\rho \right) \right]ds \\
&=  h(x)  +\int_r^{t} a^{ij}(s)  \bE'\left[h_{x^ix^j} \left(x+\int_0^{t-s} \sigma'(t-\rho)dW'_\rho \right) \right]ds \\
&=  h(x)  +\int_r^{t} a^{ij}(s)  \bE'\left[h_{x^ix^j} \left(x+ X'_t -X'_s\right) \right]ds.
\end{align*}
This will be used later for the solution representation to SPDEs (see Remark \ref{2019041730}(ii) below). 

\end{enumerate}

\end{remark}

\mysection{Stochastic  linear equations with additive noises }

In this section, we study the following SPDE with additive noises:
\begin{align}
							\notag
& du= \left( a^{ij}(t)u_{x^ix^j}+f  \right)dt + g^k dw^k_t,\quad (t,x) \in  (0,\tau] \times \fR^d \\
							\label{linear eqn}
&u(0,x)=u_0(x),
\end{align}
where $\tau$ is a bounded stopping time.  We  assume that  the coefficients $a^{ij}$  are predictable functions of $(\omega,t)$ and satisfy 
 \begin{align}
						\label{2019041850}
 a^{ij}(t) \xi^i \xi^j \geq 0,\quad  \quad \forall  (\omega, t,\xi) \in  \Omega \times (0,\infty) \times \fR^d.
 \end{align}


We denote by $\bH_c^\infty(\tau, l_2)$ the space of stochastic processes $g=(g^1,g^2, \ldots)$ such that $g^k=0$ for all large $k$ and each $g^k$ is  of the type
$$
g^k(t,x)= \sum_{i=1}^{j(k)}1_{(\tau_{i-1},\tau_i]}(t) g^{ik}(x),
$$
where  $j(k) \in \bN$, $g^{ik} \in C_c^\infty(\fR^d)$, and $\tau_i$ are stopping times with $\tau_i \leq \tau$.
Similarly, we denote by $\bH_c^\infty( \tau )$ the space of stochastic processes $g$ such that 
$$
g(t,x)= \sum_{i=1}^{j}1_{(\tau_{i-1},\tau_i]}(t) g^{i}(x),
$$
where $j \in \bN$, $g^{i} \in C_c^\infty(\fR^d)$, and $\tau_i$ are stopping times with $\tau_i \leq \tau$.
Also, we denote by $\bH_c^\infty(\fR^d)$ the space of random variables $g_0$ of the type
$$
g_0(\omega,x)= 1_A(\omega)g(x)
$$
where  $g \in C_c^\infty(\fR^d)$, and $ A \in \rF_0$.

It is known that $\bH^\infty_c(\tau,l_2) $ is dense in $\bH^\gamma_p(\tau,l_2)$ for all $p \in (1,\infty)$ and $\gamma \in \fR$ (for instance, see \cite[Theorem 3.10]{Krylov1999}).
In particular, $\bH_c^\infty(\tau)$ is dense in $\bH^\gamma_p(\tau)$ for all $p \in (1,\infty)$ and $\gamma \in \fR$.
Following the idea of \cite[Theorem 3.10]{Krylov1999}, one can also easily check that $\bH_c^\infty(\fR^d)$ is dense in $\bB_p^\gamma$ for all $p \in (1,\infty)$ and $\gamma \in \fR$.

\begin{theorem}
						\label{thm 20181222}
Let $p \in [2, \infty)$, $T \in [0,\infty)$,  $\gamma \in \fR$, and $\tau$ be a stopping time such that $\tau \leq T$.
Assume that the coefficients $a^{ij}(t)$ are locally integrable in $t$, that is,
\begin{align}
						\label{2018122801}
\sum_{i,j}\int_0^t  |a^{ij}(\omega,t)|  dt < \infty,  \quad \forall t ~ (a.s.) .
\end{align}
Then for all 
$u_0 \in \bH_p^{\gamma}$, $f \in \bH^\gamma_p \left(\tau\right)$, and $g \in \bH^{\gamma}_p \left(\tau, l_2\right)$, 
there exists a unique solution 
$u \in L_p \left( \Omega, \rF; C\left( [0,\tau] ; H^\gamma_p \right) \right)$ to \eqref{linear eqn} such that
\begin{align}
                             \label{2018122210}
\bE \sup_{t  \in [0,\tau] } \|u(t,\cdot)\|^p_{H^\gamma_p}
\leq N(p,T) \left( \|u_0\|^p_{\bH^\gamma_p} + \|f\|^p_{\bH^\gamma_p(\tau)} 
+\|g\|^p_{\bH^\gamma_p(\tau,l_2)}  \right).
\end{align}
\end{theorem}
\begin{proof}
If the coefficients are bounded then the results were proved  in \cite{krylov1986characteristics}. 

Due to the isometry of the map $(1-\Delta)^{\gamma/2}$ on $H^{\gamma}_p$, we may assume  $\gamma=0$.

{\bf Part I.} (A priori estimate and the uniqueness).

First, we show that any solution $u \in L_p ( \Omega; C( [0,\tau] ;L_p)) $ to \eqref{linear eqn} satisfies \eqref{2018122210}
 following the proof of Lemma \ref{deter a priori}, but using It\^o's formula instead of the chain rule.
Let $u$ be a solution to \eqref{linear eqn} in $ L_p(\Omega; C( [0,\tau] ; L_p)) $.
Due to the argument of Sobolev's mollifier used in the proof of Lemma \ref{deter a priori}, we may assume that the given solution $u$ and the data $f$ and $g$  are sufficiently smooth with respect to $x$. 
Then  by It\^o's formula,
\begin{align*}
d\left(|u|^p \right) 
&=  p |u|^{p-2} u  \left( a^{ij}(t) u_{x^ix^j} + f \right)  dt + p |u|^{p-2}u  u_x g^k dw^k_t \\
&\qquad +\frac{1}{2} p (p-1) |u|^{p-2} |g|_{l_2}^2 dt.
\end{align*}
Applying (stochastic) Fubini's theorem, and the integration by parts, we have
\begin{align*}
&\int_{\fR^d}|u(t,x)|^p  dx \\
&= \int_{\fR^d}|u_0(x)|^pdx
-\int_0^t \int_{\fR^d} p(p-1)|u|^{p-2}(s,x) u_{x^j}(s,x)a^{ij}(t)u_{x^i}(s,x)ds dx \\
&\quad + \int_0^t \int_{\fR^d} p|u|^{p-2}(s,x) u(s,x) f(s,x)dxds \\
& \quad + \frac{1}{2} p(p-1)\int_0^t \int_{\fR^d} |u|^{p-2}(s,x)|g|_{l_2}^2 (s,x)dxds \\
&\quad + \int_0^t \int_{\fR^d} p|u|^{p-2}(s,x) u(s,x) (g^{k})  (s,x)dxdw_s^k \qquad \forall  t \in [0,\tau].
\end{align*}
By the BDG (Burkholder-Davis-Gundy) inequality,  the H\"older inequality, and the generalized Minkowski inequality,
\begin{align*}
& \bE \left[\sup_{t \leq \tau} \left|\int_0^\tau \int_{\fR^d} p|u|^{p-2}(s,x) u(s,x) (g^{k})  (s,x) dxdw_s^k \right| \right]\\
&\leq N(p) \bE \left[\left|\int_0^\tau \left|\int_{\fR^d} |u|^{p-2}(s,x) u(s,x) (g^{k})  (s,x) dx \right|_{l_2}^2 ds \right|^{1/2} \right] \\
&\leq N(p) \bE \left[\left|\int_0^\tau \left|\int_{\fR^d} |u|^{p-1}(s,x)  |g  (s,x)|_{l_2} dx \right|^2 ds \right|^{1/2} \right] \\
&\leq N(p) \bE \left[\left|\int_0^\tau  \left| \|u^{p-1}(s,\cdot)\|_{L_q}    \|g\|_{L_p(l_2)}  \right|^2 ds \right|^{1/2} \right],
\end{align*}
where $q= \frac{p}{p-1}$.
Due to \eqref{2019041850},
\begin{align*}
\int_0^t \int_{\fR^d} p(p-1)|u|^{p-2}(s,x) u_{x^j}(s,x)a^{ij}(t)u_{x^i}(s,x)ds dx  \geq 0.
\end{align*}
Thus
\begin{align*}
&\bE \left[ \sup_{t \in [0,\tau]} \int_{\fR^d}|u(t,x)|^p  dx \right] \\
&\leq \bE \left[ \int_{\fR^d}|u_0(x)|^pdx \right]
+\bE\left[ \left|\int_0^\tau \int_{\fR^d} p|u|^{p-2}(s,x) u(s,x) f(s,x)dxds  \right| \right] \\
&\qquad + N\bE \left[\left|\int_0^\tau  \left| \|u^{p-1}(s,\cdot)\|^{}_{L_q}    \|g\|_{L_p(l_2)}  \right|^2 ds \right|^{1/2} \right].
\end{align*}
By H\"older's  inequality and Young's inequality,   for any constant $c>0$
\begin{align*}
&\bE\left[ \left|\int_0^\tau \int_{\fR^d} |u|^{p-2}(s,x) u(s,x) f(s,x)dxds  \right|\right]  \\
&\leq \bE \left[ \int_0^\tau \|  c u^{p-1}(s,\cdot)\|_{L_q(\fR^d)}  \| c^{-1}f\|_{L_p(\fR^d)} ds \right] \\
&\leq \frac{p}{q} c^q  \bE \left[\int_0^\tau    \| u (s,\cdot)\|^p_{L_p(\fR^d)}  ds \right] +   c^{-p} \frac{1}{p} \bE \left[\int_0^\tau  \| f\|^p_{L_p(\fR^d)} ds \right] \\
&\leq \frac{p}{q} c^q T  \bE \left[ \sup_{s \leq T} \| u (s,\cdot)\|^p_{L_p(\fR^d)}  \right] +   c^{-p} \frac{1}{p} \bE \left[ \int_0^\tau  \| f\|^p_{L_p(\fR^d)} ds \right].
\end{align*}
Similarly,  
\begin{align*}
&\bE \left[\left|\int_0^\tau  \left| \| u^{p-1}(s,\cdot)\|_{L_q}    \|g\|_{L_p(l_2)}  \right|^2 ds \right|^{1/2} \right] \\
&=  \bE \left[\left|\int_0^\tau  \left| \| c u^{p-1}(s,\cdot)\|_{L_q}    \| c^{-1} g\|_{L_p(l_2)}  \right|^2 ds \right|^{1/2} \right] \\
&\leq  \bE \left[   \sup_{t \leq \tau} \| c^{\frac{1}{p-1}} u(t,\cdot)\|^{p-1}_{L_p} \left|\int_0^\tau  \left|    \| c^{-1} g\|_{L_p(l_2)}  \right|^2 ds \right|^{1/2} \right] \\
&\leq   \left( \bE \left[   \sup_{t \leq \tau} \| c^{\frac{1}{p-1}}u(t,\cdot)\|^p_{L_p}  \right] \right)^{1/q}  \left(\bE\left[ \left|\int_0^\tau  \left|    \| c^{-1} g\|_{L_p(l_2)}  \right|^2 ds \right|^{p/2} \right] \right)^{1/p}\\
&\leq  N\left( \bE \left[   \sup_{t \leq \tau} \| c^{\frac{1}{p-1}} u(t,\cdot)\|^p_{L_p}  \right] \right)^{1/q}   \left( \| c^{-1} g\|^p_{\bL_p(\tau)} \right)^{1/p} \\
&\leq   N \left(\frac{c^{\frac{p}{p-1}} }{q} \bE \left[   \sup_{t \leq \tau} \|  u(t,\cdot)\|^p_{L_p}  \right]     + \frac{1}{c^p p} \left( \| g\|^p_{\bL_p(\tau)} \right)^{1/p} \right).
\end{align*}
Therefore taking $c>0$ small  enough, 
we obtain 
\begin{align*}
\bE\left[ \sup_{t \in [0,T]} \int_{\fR^d}|u(t,x)|^p  dx \right]
\leq N\left( \|u_0\|^p_{\bL_p} + \|f\|_{\bL_p(\tau)}^p + \|g\|_{\bL_p(\tau,l_2)}^p \right),
\end{align*}
where $N$ depends only on $p$ and $T$. 
Obviously, this a priori estimate yields the uniqueness of the solution.

\bigskip

{\bf Part II.} (Existence)
\smallskip
We divide the proof of the existence into several steps.
\begin{enumerate}[(i)] 
\item First, we assume that $u_0=0$, $f=0$,  and
\begin{align}
					\label{2019032140}
g \in \bH_c^\infty(\tau, l_2).
\end{align}
For a while, we additionally assume that there exists a positive constant $M>0$ such that
\begin{align}
						\label{20190422201}
\int_0^t  |a(t)|  dt  \leq M,  \quad \forall t\leq \tau ~ (a.s.).
\end{align}

Let $(\Omega', \rF', P')$  be a probability space different from $(\Omega,\cF,P)$ and  $W'_t$ be a Wiener process on $(\Omega', \rF', P')$.   
Take  a symmetric matrix-valued process $\sigma'_t$ on $\Omega$  such that 
$$
( a^{ij}(\omega,t))=\frac{1}{2} (  \sigma')^2(\omega,t).
$$
For each fixed $\omega\in  \Omega$, we define the stochastic process $X'_{t,\omega}$ on  $\Omega'\times [0,\infty)$  by 
\begin{equation}
   \label{eqn 4.25.3}
X'_{t,\omega}=\int^t_0  \sigma'(\omega,t)dW'_t,
\end{equation}
where $W'_t$ is a Wiener process on a probability space $(\Omega', \rF', P')$.
Then by \cite[IV,Theorem 63]{Pro},  the process $X'_{t,\omega}$ has a  $\rF' \otimes \rF \otimes \cB([0,\infty))$-measurable version and predictable  for each fixed $\omega$.
Set
$$
v(t,x) := \int_0^t g^k(s,x)dw^k_s,
$$
$$
y=y_{\omega}(t,x) := a^{ij}(\omega,t) v_{x^ix^j}(\omega,t,x) ,
$$
and  for each $\omega$, consider the deterministic PDE
\begin{align}
						\notag
&z_t(t,x) = a^{ij}(\omega,t) z_{x^ix^j}(t,x) + y_{\omega} (t,x), \qquad (t,x) \in (0,\tau(\omega)] \times \fR^d \\
						\label{2018122310}
&z(0,x)=0.
\end{align}
Note that by the Fubini Theorem, the BDG inequality, and the  H\"older inequality,
\begin{align}
							\notag
\bE  \left[  \sup_{t \leq \tau}  \int_{\fR^d}\left| v_{xx}(t,x)  \right|^p dx   \right]
							\notag
&\leq N\bE  \left[   \int_{\fR^d}\left|  \int_0^\tau |g_{xx}(t,x)|^2_{l_2}   dt\right|^{p/2} dx   \right] \\
							\label{2019041701}
&\leq N\|g_{xx}\|^p_{\bL_p(\tau, l_2)} < \infty,
\end{align}
and similarly,
\begin{align*}
\bE  \left[  \sup_{t \leq \tau}  \int_{\fR^d}\left| v(t,x)  \right|^p dx   \right]
\leq N\|g\|^p_{\bL_p(\tau, l_2)} < \infty.
\end{align*}
Note that 
\begin{eqnarray}
\nonumber
   \int_0^\tau \|a^{ij}(t)v_{xx}(t,\cdot) \|_{L_p}  dt  &\leq&  \int_0^\tau |a(t)|dt  \cdot  \sup_{t \leq \tau} \| v_{xx}(t, \cdot)\|_{L_p}       \\   
    &\leq& M 
   \sup_{t \leq \tau} \| v_{xx}(t, \cdot)\|_{L_p}.	 \label{eqn 4.25.1}						
   \end{eqnarray}
Applying \eqref{2019041701} and  H\"older's inequality,  we have
\begin{align}
								\notag
\bE  \left[   \int_0^\tau \|a^{ij}(t)v_{xx}(t,\cdot) \|_{L_p}  dt \right]
\leq  M \left(\bE \left[    \sup_{t \leq \tau} \| v_{xx}(t, \cdot)\|^p_{L_p}  \right]  \right)^{1/p} < \infty.
\end{align}
Thus, $y_{\omega}\in L_1((0,\tau(\omega)]; L_p)$ (a.s.), and  by Remark \ref{no smooth}(i),
\begin{equation}
   \label{eqn 4.29.1}
z_{\omega}=z_{\omega}(t,x):= \int_0^t    \bE'[y_{\omega}(s,x+X'_{t,\omega}-X'_{s,\omega}) ] ds
\end{equation}
is a solution to \eqref{2018122310} such that $z_{\omega}\in C\left( [0,\tau] ; L_p \right)$, and 
$$\sup_{t\leq \tau}\|z_{\omega}\|_p\leq N\int^{\tau}_0 \|y_{\omega}(t)\|_p \,dt.
$$
This, \eqref{eqn 4.25.1},    \eqref{2019041701},  and   \eqref{eqn 4.29.1} yield that  $z$ is $\cF_t$-adapted,   $L_p$-valued predictable, and 
$$
z:=z_{\omega}(t,x)\in L_p( \Omega, \rF; C( [0,\tau] ; L_p)).
$$
Define 
\begin{align}
					\notag
u(t,x) 
&:= z(t,x) + v(t,x) \\
						\label{2018122213}
&=\int_0^t    \bE'[y_{\omega}(s,x+X'_{t,\omega}-X'_{s,\omega}) ] ds +\int_0^t g^k(s,x)dw^k_s,
\end{align}
where
$$
y(t,x) :=y_{\omega}(t,x) =a^{ij}(\omega,t) v_{x^ix^j}(\omega,t,x) .
$$
Then
\begin{align*}
u(t,x)
&= \int_0^t \left( a^{ij}(s) z_{x^ix^j}(s,x) +y(s,x)  \right)ds + v(t,x) \\
&=\int_0^t a^{ij}(s) u_{x^ix^j}(s,x) ds + \int_0^tg^k(s,x) dw_s^k.
\end{align*}
Therefore  $u$ becomes a solution to \eqref{linear eqn}  in  $L_p( \Omega,\rF; C( [0,\tau] ; L_p))$
if \eqref{20190422201} holds.

To remove the bounded condition \eqref{20190422201},  consider stopping times
$$
\tau_n :=  \inf \left\{ t \leq  \tau  :  \sum_{i,j}\int_0^t  |a^{ij}(t)|  dt > n  \right\}.
$$
Then, since  \eqref{20190422201} holds with $\tau_n$ and $n$, by the above result
 there exists a  solution $u_n$ to \eqref{linear eqn}  with $\tau_n$ in  $L_p(\Omega; C( [0,\tau_n] ; L_p)) $.
By  the uniqueness of a solution and a priori estimate \eqref{2018122210}  obtained in Part I,
$u_n = u_m$ $a.e.$ on  $\{(\omega,t) : t \in [0, \tau_n(\omega)] \}$ for all $m \geq n$, 
and
\begin{align}
							\label{2019042250}
\bE \sup_{t  \in [0,\tau_n] } \|u_n(t,\cdot)\|^p_{L_p}
\leq N(p,T) 
\bE \|g\|^p_{\bL_p(\tau,l_2)}.
\end{align}
Define
$$
\tilde u(t) := \lim_{n \to \infty} u_n(t),
$$
where the limit is the point-wise limit  on a  subset of  $\{(\omega,t) :  t \in [0, \tau(\omega))\}$.
Since $\tau_n \to \tau$ $(a.s.)$ as $n \to \infty$,  we have $\tilde u\in C([0,\tau); L_p)$ (a.s.), and $\tilde u$ becomes a (distribution-valued) solution to \eqref{linear eqn} for $t<\tau$.  Also, since $\tilde u=u_n$ for $t\leq \tau_n$, we have
$$
\sup_{t\leq \tau_n}\|u_n(t)\|^p_p=\sup_{t\leq \tau_n}\|\tilde u(t)\|^p_p,
$$
and therefore, applying Fatou's lemma to \eqref{2019042250}, we conclude
\begin{align*}
\bE \sup_{t  \in [0,\tau) } \|\tilde u(t,\cdot)\|^p_{L_p}
\leq N(p,T)
\bE\|g\|^p_{\bL_p(\tau,l_2)}.
\end{align*}
Note also that, since $u_n$ is defined as in \eqref{2018122213} for $t\leq \tau_n$, it follows that if $t<\tau$ then $\tilde u$ is equal to the right hand side of   \eqref{2018122213}, which is  adapted and continuous $L_p$-valued process on $[0,\tau]$. Therefore, we conclude that   there exists  a continuous extension $u$  which is a version of $\tilde u$ and a solution to  \eqref{linear eqn}  in  the class $L_p( \Omega; C( [0,\tau] ; L_p))$.
\smallskip

\item Second, we assume 
$$
u_0 \in \bH_c^\infty(\fR^d), \quad f \in \bH_c^\infty(\tau), \quad  g \in \bH_c^\infty( \tau , l_2).
$$
For each $\omega$, consider the equation
\begin{align}
						\notag
& z_t(t,x) =  a^{ij}(t)z_{x^ix^j}(t,x)+f(t,x)
						\label{2019032110}
\quad (t,x) \in  (0,\tau(\omega)] \times \fR^d  \\
&z(0,x)=u_0(x).
\end{align}
Take $X'_{t,\omega}$ from \eqref{eqn 4.25.3}, and define
\begin{align}
						\label{2019010610}
z_{\omega}(t,x)= \bE'[u_0(\omega,x+X'_{t,\omega})] + \int_0^t\bE'[f(\omega, s,x+X'_{t,\omega}-X'_{s,\omega} )]ds.
\end{align}
Then by Remark \ref{no smooth}(i), $z_{\omega}$ is a solution to \eqref{2019032110} and $z_{\omega}\in C\left( [0,\tau(\omega)] ; L_p \right)$ for each $\omega$.
Moreover, by the generalized Minkowski inequality,
$$
z:=z_{\omega}(t) \in L_p \left( \Omega, \rF ; C\left( [0,\tau] ; L_p \right)  \right).
$$
Moreover, due to (i), formula \eqref{2018122213} gives a unique solution $\bar{v}$ 
to the equation
\begin{align*}
& d\bar v=a^{ij}(t) \bar v_{x^ix^j}dt + g^k dw^k_t
\quad (t,x) \in  (0,\tau ] \times \fR^d  \\
&\bar v (0,x)=0
\end{align*}
in $ L_p \left( \Omega, \rF ; C\left( [0,\tau] ; L_p \right)  \right)$.
Then considering $u:= z+ \bar v$, we finally find a solution 
  to equation \eqref{linear eqn} in the class $L_p \left( \Omega, \rF ; C\left( [0,\tau] ; L_p \right)  \right)$.

\item (General Case)
Choose sequences $u_0^n \in \bH_c^\infty(\fR^d)$, $f^n \in \bH_c^\infty(\tau)$, and $g^n \in\bH_c^\infty(\tau,l_2)$ so that
$$
u^n_0 \to u_0 \quad \text{in} \quad \bL_p, \qquad f^n \to f \qquad \text{in} \quad\bL_p(\tau), \qquad g^n \to g \qquad \text{in} \quad\bL_p(\tau,l_2)
$$
as $n \to \infty$. 
Then for each $n$, by (ii) there exists a solution 
$$
u^n \in L_p \left( \Omega, \rF ; C\left( [0,\tau] ; L_p \right)  \right)
$$
to the equation 
 \begin{align*}
&d u^n (t,x)= \left(a^{ij}(t)u^n_{x^ix^j}(t,x)+f^n(t,x)  \right) dt  +(g^n)^k(t,x) dw^k_t\\
&u^n(0,x)=u^n_0(x) \qquad (\omega, t,x) \in  \Omega \times (0, \tau]  \times \fR^d .
 \end{align*}
and thus for all $n,m$
 \begin{align*}
&d (u^n -u^m) (t,x)= \left(a^{ij}(t)(u^n-u^m)_{x^ix^j}(t,x)+(f^n-f^m)(t,x)  \right) dt   \\
&\qquad  \qquad \qquad  \qquad + (g^n-g^m)^k(t,x) dw^k_t \qquad (\omega, t,x) \in  \Omega \times (0,\tau] \times \fR^d \\
&(u^n-u^m)(0,x)=(u^n_0 -u^m_0)(x)  .
 \end{align*}
\end{enumerate}
Due to   a priori estimate \eqref {2018122210}, we have
\begin{align*}
& \bE \sup_{t  \in [0,\tau] } \|(u^n-u^m)(t,\cdot)\|^p_{L_p} \\
&\leq N(p,T) \left(\|f^n -  f^m\|^p_{\bL_p(\tau)} 
+\|  g^n - g^m\|^p_{\bL_p(\tau,l_2)} + \|u^n_0 - u^m_0\|^p_{\bL_p} \right)
\end{align*}Thus $u^n$ becomes a Cauchy sequence in $L_p \left( \Omega, \rF ; C\left( [0,\tau] ; L_p \right)  \right)$ 
and by taking the limit, we have a solution   $u \in L_p \left( \Omega, \rF ; C\left( [0,\tau] ; L_p \right)  \right) $ to equation \eqref{linear eqn}.
The theorem is proved. 
\end{proof}

The results of the following remark   will  not used anywhere in this article.

\begin{remark}
						\label{2019041730}

(i)  Based on  the approximation  used in the above proof,  one can easily check that  if $u_0 \in \bL_p$, $f \in \bL_p \left(\tau\right)$, and $g \in \bH^2_p \left(\tau, l_2\right)$ then the solution  to the equation 
\begin{align*}
 du= \left( a^{ij}(t)u_{x^ix^j}+f  \right)dt + g^k dw^k_t, \quad t>0; \quad u(0,x)=u_0
\end{align*}
is given by 
\begin{align}
							\notag
u (t,x)
&=\bE'[u_0(\omega,x+X'_{t,\omega})] + \int_0^t\bE'[f(\omega,s,x+X'_{t,\omega}-X'_{s,\omega} )]ds \\
						\label{sol repre}
&\quad +\int_0^t    \bE'[y_{\omega}(s,x+X'_{t,\omega}-X'_{s,\omega}) ] ds +\int_0^t g^k(s,x)dw^k_s.
\end{align}
The additional assumption $g \in \bH^2_p \left(\tau, l_2\right)$ is needed to make sense of $y_{\omega}$ which is defined by
$$
y_{\omega}(t,x) := a^{ij}(t) \int_0^t g_{x^ix^j}^k(s,x)dw^k_s.
$$
%

(ii)
If coefficients $a^{ij}(t)$ are not random, then for any $u_0 \in \bL_p$, $f \in \bL_p \left(\tau\right)$, and $g \in \bL_p \left(\tau, l_2\right)$ then   solution to the equation 
\begin{align*}
 du=  (a^{ij}(t)u_{x^ix^j}+f) \, dt + g^k \,dw^k_t, \quad t>0; \quad u(0,\cdot)=u_0
\end{align*}
is given by
\begin{align}
						\notag
u(t,x) 
&=  \bE'[u_0(x+X'_t)] + \int_0^t\bE'[f(s,x+X'_t-X'_s) ] ds  \\
						\label{eqn 4.25.5}
&\qquad + \int_0^t\bE'[g^k(s,x+X'_t-X'_s) ] dw^k_s.
\end{align}
%
Actually this is a well known  result if the coefficients are bounded  and have uniform ellipticity condition. 
The  general case can be proved based on  Ito's formula. 
For simplicity, we  only consider the case $u_0=0$ and $f=0$.  Considering an approximation argument we may assume $g\in \bH^2_p(\tau,l_2)$. This is possible because there are no derivatives of $g$ in formula \eqref{eqn 4.25.5}.

Using \eqref{sol repre} and applying the integration by parts, the Fubini Theorem, and the stochastic Fubini theorem, we have
\begin{align}
						\notag
& \int_{\fR^d} \int_0^t    \bE'[y(s,x+X'_t-X'_s) ] ds\phi(x) dx  \\
						\notag
&=\int_0^t  \int_{\fR^d} a^{ij}(s)  \bE' \left[ \int_0^s g^k(r,x +X'_t-X'_s)  dw^k_r \right]  \phi_{x^i x^j}(x) dx ds  \\
						\notag
&= \int_0^t \int_0^t  1_{r<s} \int_{\fR^d} a^{ij}(s)  \bE' \left[  g^k(r,x +X'_t-X'_s)  \right]  \phi_{x^i x^j}(x) dx ds   dw^k_r \\
						\label{2018122301}
&= \int_0^t  \int_{\fR^d} \int_r^t  a^{ij}(s)  \bE' \left[  g_{x^i x^j}^k(r,x +X'_t-X'_s)  \right] ds \phi(x) dx    dw^k_r 
\end{align}
for all $t \in [0,\tau]$ $(a.s.)$.
By It\^o's formula (cf. Remark \ref{no smooth}(ii)), for all $t \geq r$ and $\omega$, we have
\begin{align}
						\label{2019041760}
\bE'[g^k(r,x+X'_t-X'_r)]
=  g^k(r,x) +\int_r^t a^{ij}(s)  \bE'\left[g^k_{x^ix^j}(r,x+X'_t-X'_s) \right]ds.
\end{align}
Thus from \eqref{2018122301}, \eqref{2019041760}, and the stochastic Fubini theorem, we have
\begin{align*}
& \int_{\fR^d} u(t,x) \phi(x)dx \\
& \int_{\fR^d} \left(\int_0^t    \bE'[y(s,x+X'_t-X'_s) ] ds +\int_0^t g^k(s,x)dw^k_s\right)\phi(x) dx  \\
&=\int_{\fR^d} \int_0^t  \left(\bE'[g^k(r,x+X'_t-X'_r)] \right) dw_r^k\phi(x)dx 
\end{align*}
for all $t\leq \tau$ (a.s.). Hence, the claim is proved. 
\end{remark}


From now on, we focus on higher regularity of solution to  equation \eqref{linear eqn}. 

\begin{lemma} \label{lemma 4.26.1}
Suppose there are constants $\kappa, M>0 $ such that 
\begin{align}
					\label{strong 2}
|a^{ij}(t)|   \leq M,   \quad \forall\, \omega\in \Omega,t >0
\end{align}
and
\begin{align}
				\label{strong 1}
 a^{ij}(t) \xi^i \xi^j \geq  \kappa |\xi|^2, \quad \forall  \omega\in \Omega, t>0,\xi \in \fR^d.
\end{align}
Let $p \geq 2$, $\tau$ be a stopping time,  $u_0\in \bB^{2-2/p}_p$, $f\in \bL_p(\tau)$ and $g\in \bH^1_p(\tau,l_2)$.  Then equation \eqref{linear eqn} has a unique solution $u\in \cap_{T>0} L_p \left( \Omega, \rF ; C\left( [0,\tau \wedge T] ; L_p \right)  \right)$, and for this solution we have  
\begin{align}
				\label{classical est}
\|u_{xx}\|_{\bL_p( \tau)} \leq N \left(  \|u_0\|_{ \dot \bB_p^{2-2/p}} + \|f\|_{\bL_p( \tau )} + \|g_x\|_{\bL_p( \tau,l_2)} \right),
\end{align}
where   $N=N(d,p,\kappa, M)$ is independent of $\tau$.
\end{lemma}
\begin{proof}
The existence and uniqueness are consequence of Theorem \ref{thm 20181222}. Estimate \eqref{classical est} was proved
 by Krylov (\cite{Krylov1999,Krylov2000}), however we give some details below because Krylov used $ \bH^{2-2/p}_p$ for the space of  initial data in place of  $\bB^{2-2/p}_p$. 

 Step 1.  Let $u_0=0$.   Then,  by  \cite[Theorem 2.1]{Krylov2000},  for any $T>0$,
   \begin{equation}
   \label{eqn 4.25.6}
\|u_{xx}\|_{\bL_p(\tau \wedge T)} \leq N(d,p,\kappa , M) \left(  \|f\|_{\bL_p(\tau \wedge T)}  +\|g_x\|_{\bL_p( \tau \wedge T,l_2)} \right),
\end{equation}
and thus one gets  \eqref{classical est} by taking $T\to \infty$.  

Step 2. In general,  take a solution  $v$ (cf. \cite[Theorem 5.1]{Krylov1999}) to equation   
$$
dv=\Delta v\,dt, \quad t>0; \quad v(0,x)=u_0
$$
such that $v\in \bH^2_p(\tau \wedge T) \cap L_p \left( \Omega, \rF ; C\left( [0,\tau \wedge T] ; L_p \right)  \right)$ for any $T>0$. Then using  a classical result in PDE (see e.g. \cite{LA} ) for each $\omega$, 
$$
 \|v_{xx}\|^p_{L_p(\tau \wedge T)}\leq N(d,p)\|u_0\|^p_{\dot B_p^{2-2/p}}.
 $$
 Thus, taking the expectation and  letting $T\to \infty$, we get 
 \begin{equation}
    \label{eqn 4.25.7}
 \|v_{xx}\|^p_{\bL_p(\tau)}\leq N(d,p)\|u_0\|^p_{ \dot \bB_p^{2-2/p}}.
 \end{equation}
 Finally, note that $\bar{u}:=u-v \in  \cap_{T>0} L_p \left( \Omega, \rF ; C\left( [0,\tau \wedge T] ; L_p \right)  \right)$ and satisfies
 $$
 d\bar{u}=(a^{ij}\bar{u}_{x^ix^j}+\bar{f})dt+ g^k dw^k_t, \quad t>0; \quad \bar{u}(0,x)=0,
 $$
 where $\bar{f}:=a^{ij}v_{x^ix^j}-\Delta v+f$.  By the result of Step 1 and \eqref{eqn 4.25.7},
 \begin{eqnarray*}
 \|u_{xx}\|_{\bL_p(\tau)} &\leq& \|\bar{u}_{xx}\|_{\bL_p(\tau)}+\|v_{xx}\|_{\bL_p(\tau)}\\
 &\leq& N(d,p,\kappa,M) \left( \|\bar{f}\|_{\bL_p(\tau)}+\|g_x\|_{\bL_p( \tau,l_2)} \right) +\|v_{xx}\|_{\bL_p(\tau)}\\
 &\leq&N(d,p,\kappa,M) \left( \|f\|_{\bL_p(\tau)}+\|g_x\|_{\bL_p( \tau,l_2)} +\|u_0\|_{\dot{\bB}_p^{2-2/p}} \right).
 \end{eqnarray*}
  The lemma is proved.
 \end{proof}

In the following lemma we show that the boundedness of coefficients is not needed for estimate  \eqref{classical est}.
 \begin{lemma}
				    \label{lemma classic}
Let $p \in [2,\infty)$,  $\tau$ be a stopping time, 
 $u_0 \in \bB_p^{2 \left(1-1/ p \right)}$, $f\in \bL_p(\tau)$, $g \in \bH_p^1(\tau,l_2)$, and $u \in \bigcap_{T>0} L_p \left( \Omega, \rF ; C\left( [0,\tau \wedge T] ; L_p \right)  \right)$ be a solution to  \eqref{linear eqn}.
Assume that \eqref{strong 1} and \eqref{2018122801} hold and coefficients $a^{ij}$  are predictable for all $i, j$.
Then there exists a positive constant $N=N(d,p)$ such that
 \begin{align}
					\label{constant}
 \|u_{xx}\|_{\bL_p(\tau)}
 \leq N   \left(  \kappa^{-1/p}\|u_0\|_{ \dot \bB_p^{2 \left(1-1/ p \right)}}  +  \kappa^{-1} \|f\|^p_{\bL_p(\tau)}  + \kappa^{-1/2} \|g_x\|_{\bL_p(\tau,l_2)} \right).
 \end{align}
  \end{lemma}

  \begin{proof}
Due to the approximation used in Theorem \ref{thm 20181222}, we may assume 
$$
u_0 \in \bH_c^\infty(\fR^d),  \quad f \in \bH_c^\infty(\tau), \quad \text{and} \quad  g \in\bH_c^\infty(\tau,l_2).
$$
We use the idea in the proof of \cite[Theorem 4.10]{Krylov1999}.
\vspace{3mm}

{\bf Step 1}.  Assume  $A(t)=(a^{ij}(t)) = \frac{\kappa}{2}  I_{d\times d}$, where  $I_{d\times d}$ is the $d \times d$ identity matrix.
Thus,  $u$ is a solution to
\begin{align*}
 du(t,x) = \left(\frac{\kappa}{2} \Delta u + f \right)dt +   g^k\,dw^k_t, \quad 0< t\leq \tau; \quad u(0,x)=u_0. 
\end{align*}
Define 
$$
\bar v(t,x) = u(t,   \sqrt \kappa x).
$$
Then $\bar v$ satisfies
\begin{align*}
 d\bar v = \left(  \frac{1}{2} \Delta \bar v (t,x)+ f( t, \sqrt \kappa x)  \right)dt +  g^k(t, \sqrt \kappa x) dw^k_t, \quad  0<t\leq \tau
 \end{align*}
 with initial data
$\bar v(0,x)=u_0( \sqrt \kappa x)$. 
Since $ \frac{1}{2} I_{d \times d}$ satisfies both \eqref{strong 1} and \eqref{strong 2} with $\kappa=M=\frac{1}{2}$,  by \eqref{classical est} applied to $\bar{v}$,   we get
$$
\kappa \|u_{xx}\|_{\bL_p (\tau)} \leq N(d,p) \left(\|f\|_{\bL_p(\tau)}+ \kappa^{1/2}\|g_x\|_{\bL_p(\tau,l_2)}+\kappa^{1-1/p}\|u_0\|_{\dot{\bB}^{2-2/p}_p}\right),
$$
which certainly leads to \eqref{constant}.

\smallskip

{\bf Step 2}. (General case).

Let  $W'_t$ be a $d$-dimensional  Wiener process on a probability space $(\Omega', \rF', P')$ different from $(\Omega, \rF, P)$.
Consider the product probability space $(\Omega \times \Omega' , \rF \times \rF', P \times P' )$. Denote
$$
\bar{\rF}_t:=\{A \times \Omega': A \in \rF_t\}, \quad \bar{\sigma}(W'_s:s\leq t):=\{\Omega \times B: B\in \sigma(W'_s:s\leq t)\},
$$
and by $\hat{\rF}_t$ we denote the smallest $\sigma$-field on $\Omega\times \Omega'$ containing above two $\sigma$-fields, that is 
$$
\hat \rF_t := \bar{\rF}_t   \vee\bar{\sigma}(W'_s : s \leq t ).
$$
 Considering $\cap_{s>t}\hat{\rF}_s$ in place of $\hat{\rF}_t$, we may assume that $\hat \rF$ satisfies  the usual condition.  
 For a stopping time $\hat{\tau}$ relative to $\hat{\rF}_t$, we define the corresponding Banach spaces $\hat{\bH}^{\gamma}_p(\hat{\tau})$, $\hat{\bL}_{p}(\hat{\tau})$, 
 $\hat{\bH}^{\gamma}_p(\hat{\tau},l_2)$,   $\hat{\bL}_{p}(\hat{\tau}, l_2)$, $\hat{\bB}^{\gamma}_p$, and $\hat{ \dot \bB}^{\gamma}_p$.  
 Then, since any stochastic process defined on $\Omega$ can be considered as a stochastic processes on $\Omega \times \Omega'$,   for any stopping $\tau$ relative to $\rF_t$ we have 
   $$
 \bH^{\gamma}_p(\tau)\subset \hat{\bH}^{\gamma}_p(\tau), \quad  \bH^{\gamma}_p(\tau, l_2)\subset \hat{\bH}^{\gamma}_p(\tau, l_2), \quad \bB^{\gamma}_p \subset \hat{\bB}^{\gamma}_p,
\quad \dot \bB^{\gamma}_p \subset \hat{ \dot \bB}^{\gamma}_p .
 $$
 
 Write
$$
A(t) = (a^{ij}(t)), \quad 
 \bar A(t) := (\bar a^{ij}(t)) := \left(a^{ij}(t) - \frac{\kappa}{2} \delta^{ij}\right),
 $$
where $\delta^{ij}$ denotes the Kronecker delta.  
Then
  $$
  A(t)= \left(A(t)-\frac{\kappa}{2}I_{d\times d} \right)+\frac{\kappa}{2}I_{d\times d}= \bar{A}(t)+\frac{\kappa}{2}I_{d\times d}.
  $$ 
  Recall that $W'_t(\omega')$ and $w^k_t(\omega)$ can be considered as Wiener processes relative to $\hat{\rF}_t$.  Take a $d\times d$ symmetric matrix $\bar{\sigma}(\omega,t)$ such that $2\bar{A}=(\bar{\sigma})^2$.  Then, since $\bar{\sigma}$ can be considered as a predictable process defined on $\Omega\times \Omega'\times [0,\infty)$, we can define the stochastic integral
$$
X_{t}:=\int^t_0 \bar \sigma(s)\,dW'_s.
$$
Since $L_p$-norms are translation invariant,  we have
$$
f(t,x+X_t)\in \hat{\bL}_p(\tau), \quad g(t,x+X_t)\in \hat{\bH}^1_p(\tau,l_2).
$$
 Therefore,  by Lemma \ref{lemma 4.26.1}, the equation
\begin{align}
							\notag
 dv(t,x) = \left( \frac{\kappa}{2} \Delta v (t,x)+ f\left(t,x +X_t \right)  \right)dt +   g^k\left(t,x+X_t\right) dw^k_t, \quad 0<t\leq \tau
\end{align}
with initial data $v(0,x)=u_0$ has a unique solution 
\begin{equation}
   \label{eqn 4.26.9}
v\in \bigcap_{T>0} L_p(\Omega\times \Omega', \rF \times \rF'; C([0,\tau \wedge T];L_p),
\end{equation}
 and  for this solution we have
\begin{eqnarray}
	\nonumber						
&&\|v_{xx}\|_{\hat{\bL}_p (  \tau)} 
\leq N  \left(  \kappa^{-1/p} \|u_0\|_{\hat{ \dot \bB}_p^{2-2/p}} +  \kappa^{-1}\|f\|_{\hat{ \bL}_p(\tau)} +  \kappa^{-1/2}\|g_x\|_{\hat{\bL}_p(\tau,l_2)} \right)\\
&& \quad\quad\quad =N  \left(  \kappa^{-1/p} \|u_0\|_{ \dot \bB_p^{2-2/p}} +  \kappa^{-1}\|f\|_{\bL_p(\tau)} +  \kappa^{-1/2}\|g_x\|_{\bL_p(\tau,l_2)} \right).
\label{2019041821}
\end{eqnarray}
Next, set 
$$
z(t,x):= v\left(t,x - X_t \right).
$$
Then, using \eqref{eqn 4.26.9} and  $L_p$-continuity (i.e. $\lim_{|y|\to 0} \|h(x-y)-h(x)\|_p=0$), we conclude
$$
z\in \bigcap_{T>0} L_p(\Omega\times \Omega', \rF \times \rF'; C([0,\tau \wedge T];L_p).
$$
By the It\^o-Wentzell formula (cf. \cite[Theorem 1.1]{Krylov2011}), $z$ satisfies the equation
\begin{eqnarray}
								\nonumber 
 dz(t,x) &=& \left( a^{ij}(t)z_{x^ix^j} (t,x)+ f(t,x)  \right)dt +   g^k(t,x) dw^k_t \\ &&  \,  -z_{x^i}(t,x) \bar {\sigma}^{ij} dW'^{j}_t, \quad   t<\tau ; \quad z(0,x)=u_0.
  \label{2019041810} 
  \end{eqnarray}
Let $ \hat{ \bE} \left[ \cdot | \bar{ \rF}_t \right]$ denote the conditional expectation with respect to $\bar{\rF}_t$. Note  that
\begin{equation}
   \label{eqn 4.26.7}
\hat{\bE}\left[\int^t_0 z_{x^i}(s,x)\bar{\sigma}^{ij} dW'^j_s | \bar{\rF}_t \right]=0
\end{equation}
because  the process $W'_t$ is independent of $(\bar{\rF}_r)_{r>0}$.  Denote
\begin{equation}
   \label{eqn 4.26.10}
\bar{u}(t):= \hat{ \bE} \left[ z(t) | \bar{ \rF}_t \right] \quad  \in \,\, \bigcap_{T>0} L_p(\Omega\times \Omega', \rF \times \rF'; C([0,\tau \wedge T];L_p).
\end{equation}
The inclusion above is due to  conditional Jensen's inequality. 
 Then, by  \cite[Theorem 1.4.7]{Roz}, for each $t$, 
 \begin{eqnarray}
 \nonumber
 \hat{ \bE} \left[ \int^t_0 a^{ij}(s)z_{x^ix^j}(s) \,ds | \bar{ \rF}_t \right]   &=& \int^t_0   \hat{ \bE} \left[  a^{ij}(s) z_{x^ix^j}(s) | \bar{ \rF}_s \right] \,ds \quad (a.s.)\\
 &=& \int^t_0 a^{ij}(s)\bar{u}_{x^ix^j}(s) \,ds \quad (a.s.).  \label{eqn 4.26.8}
 \end{eqnarray}
Thus,  taking the conditional expectation  to equation \eqref{2019041810}  with respect to $\bar{\rF}_t$ and using  \eqref{eqn 4.26.7}, \eqref{eqn 4.26.8}, and  \eqref{eqn 4.26.10}, we conclude  that  $\bar{u}$ satisfies
\begin{align*}
& d\bar u(t,x) = \left( a^{ij}(t)\bar u_{x^ix^j} (t,x)+ f(t,x)  \right)dt +   g^k\left(t,x\right) dw^k_t, \quad 0<t\leq \tau   \\
&\bar u (0,x)=u_0(x).
\end{align*}

In other words, both $u$ and $\bar u$ are  solutions to \eqref{linear eqn} in the class $\cap_{T>0} L_p(\Omega\times \Omega', \rF \times \rF'; C([0,\tau \wedge T];L_p)$.
By the uniqueness result of  Theorem \ref{thm 20181222}, we get $u=\bar{u}$. Therefore,
\begin{align*}
\|u_{xx} \|_{\bL_p} =\|u_{xx}\|_{\hat{\bL}_p(\tau)}
=\|\bar u_{xx} \|_{\hat{\bL}_p(\tau)}= \|z_{xx} \|_{\hat{\bL}_p(\tau)} = \|v_{xx} \|_{\hat{\bL}_p(\tau)} .
\end{align*}
This and \eqref{2019041821} finish the proof of the lemma.
   \end{proof}

\begin{lemma}
						\label{2019011301}
 Let $p \in [2,\infty)$, $\tau$ be a stopping time, $u_0 \in \bB_p^{2 \left(1-1/ p \right)}$,  $f\in  \bL_p(\tau,\delta^{1-p})$, 
 $g\in \bH^1_p(\tau, \delta^{1-p/2}, l_2)$, 
 and $u \in\bigcap_{T>0} L_p \left( \Omega, \rF ; C\left( [0,\tau \wedge T] ; L_p \right)  \right)$ be a solution to equation \eqref{linear eqn}. 
Assume that coefficients $a^{ij}(t)$ are predictable, 
\begin{align}
							\label{2019012610}
\int_0^\tau  |a^{ij}(t)| dt < \infty ~(a.s.)
\end{align}
for all $i, j \in \{1,\ldots,d\}$, and
 $$
 a^{ij}(t) \xi^i \xi^j \geq 0, \quad \forall  (\omega, t,\xi) \in  \Omega \times (0,\infty) \times \fR^d.
 $$
Then
\begin{align}
									\label{2019011110}
\|u_{xx}\|_{\bL_p(\tau,\delta)}  \leq N(d,p) \left(  \|u_0\|_{\dot \bB_p^{2 \left(1-1/ p \right)}}  +  \|  f\|_{\bL_p( \tau,\delta^{1-p} )}  + \|  g_x\|_{\bL_p( \tau,\delta^{1-p/2},l_2)} \right),
\end{align}
where $\delta(t)$ is the smallest eigenvalue of the matrix $\left(a^{ij}(t)\right)$.
       
\end{lemma}

\begin{proof}

{\bf{Step 1}}. First we assume 
$$
u_{xx} \in \bL_p( \tau ,\delta)\cap \bL_p(\tau),
$$ 
and there exists a positive constant $\varepsilon \in (0,1]$ such that

\begin{align}
						\label{2019011050}
\delta(t)\geq \varepsilon>0 \qquad \forall t, \omega.
\end{align}
For $t>0$, denote
$$
\beta(t)=\int^t_0 \delta(s)ds,
$$
and let $\psi(t)$ be the inverse of $\beta(t)$.  Then
$$
\psi(\beta(t))=t
$$
and thus
\begin{align}
						\label{2019020401}
\psi'(\beta(t) ) \beta'(t) = \psi'(\beta(t)) \delta(t)=1, \quad \forall (\omega,t).
\end{align}
Since for each fixed $\omega \in \Omega$, $\beta(t)$ is a strictly increasing  continuous function with respect to $t$, we have
$$
\psi(t)=\inf \{s \in [0,\infty) : \beta(s)>t\}.
$$
Thus for each $\omega$, $\psi(t)$ is a strictly increasing continuous function with respect to $t$ and 
$$
\beta(t)=\inf \{s \in[0,\infty) : \psi(s)>t\}.
$$
In particular, both $\psi(t)$ and $\beta(t)$ are stopping times. 
Define
$$
\tilde{\rF}_t:=\rF_{\psi(t)}, \quad m^k_t=w^k_{\psi(t)}.
$$
Then $m^k_t$ is a square integrable  continuous martingale relative to $\tilde{\rF}_t$ such that
$$
[m^k]_t=\psi(t), \quad d[m^k]_t=\psi'(t)dt=\frac{1}{\delta(\psi(t)) }dt.
$$
Thus there exist  $\tilde \rF_t$-adapted independent Wiener processes $\tilde{w}^k_t$ such that  
$$
 m^k_t:=w^k_{\psi(t)}=\int_0^t  1/ \sqrt{\delta(\psi(s))} d\tilde w^k_s.
 $$
Recall that $u$ is a solution to \eqref{linear eqn} and consider the function  $v(t,x) :=u(\psi(t),x)$.
Then $v$ satisfies
 \begin{align*}
 dv(t,x)
&= \left(a^{ij}(\psi(t))u_{x^ix^j}(\psi(t),x)\psi'(t)+ f(\psi(t),x)\psi'(t)\right) dt + g^k(\psi(t),x) dw^k_{\psi(t)}\\
&= \left(\tilde{a}^{ij}(t)v_{x^ix^j} (t,x)  + \tilde{f}(t,x)\right) dt + \tilde{g}^k (t,x) d\tilde w^k_t, \quad \quad  0<t\leq \beta(\tau), 
 \end{align*}
with initial condition $v(0,x)=u_0$,  where 
\begin{align*}
 \tilde{a}^{ij}(t) =a^{ij}(\psi(t))\psi'(t)  =a^{ij}(\psi(t)) /\delta(\psi(t) ) ,
\end{align*}
 $$
 \tilde{f}(t,x)= f(\psi(t) ,x)  \psi'(t) = f(\psi(t) ,x)  / \delta(\psi(t)),
 $$
and
 $$
 \tilde{g}(t,x)= g(\psi(t) ,x)  \sqrt{\psi'(t)} = g(\psi(t) ,x)  / \sqrt{\delta(\psi(t))}.
 $$

Since $\delta(\psi(t))$ is the smallest eigenvalue of $a^{ij}( \psi(t) )$,
\begin{align*}
 \tilde{a}^{ij}(t) \xi^i \xi^j
 =a^{ij}(\psi(t)) \frac{1}{ \delta(\psi(t))} \xi^i \xi^j \geq |\xi|^2 \qquad  \forall \xi \in \fR^d,
\end{align*}
i.e. the ellipticity constant of the coefficients $\tilde a^{ij}(t)$ is 1. 
Thus by Lemma \ref{lemma classic}, a change of variables,  and \eqref{2019020401}, 
 \begin{align*}
 \|u_{xx}\|_{\bL_p( \tau,\delta)}
& = \|v_{xx}\|_{\bL_p(\beta(\tau))} \\
& \leq N(d,p) \left(  \|u_0\|_{\dot \bB_p^{2 \left(1-1/ p \right)}}  +  \| \tilde f\|_{\bL_p( \beta(\tau))}  + \| \tilde g_x\|_{\bL_p(\beta(\tau),l_2)} \right) \\
&=  N(d,p) \left(  \|u_0\|_{\dot \bB_p^{2 \left(1-1/ p \right)}}  +  \|  f\|_{\bL_p( \tau, \delta^{1-p} )}  + \|  g_x\|_{\bL_p( \tau, \delta^{1-p/2},l_2)} \right).
 \end{align*}

\smallskip
 {\bf{Step 2}}. Second, we only assume that
 \begin{align}
						\label{2019011310}
u_{xx} \in \bL_p( \tau ,\delta) \cap \bL_p(\tau).
\end{align}
 In other words, we remove condition \eqref{2019011050} in this step.
For $\varepsilon>0$, denote 
 $$
 a^{ij}_\varepsilon(t)=a^{ij}(t)+\varepsilon I, \quad \delta_{\varepsilon}(t):=\delta(t)+\varepsilon.
 $$
 Then $u$ satisfies
 \begin{align*}
& du= \left( a_\varepsilon^{ij}(t)u_{x^ix^j}+f -\varepsilon \Delta u  \right)dt + g^k dw^k_t, \quad 0<t\leq \tau,\\
&u(0,x)=u_0(x).
\end{align*}

 By Step 1 and  the inequalities that $\delta^{1-p}_{\varepsilon}\leq \delta^{1-p}$ and $\delta_\varepsilon^{1-p/2} \leq \delta^{1-p/2}$,
 \begin{align}
							\notag
& \bE\int^\tau_0 \|u_{xx}(t,\cdot)\|^p \delta_{\varepsilon}(t) dt\\
							\notag
& \leq N \|u_0\|_{\dot \bB_p^{2 \left(1-1/ p \right)}}  \\
&\qquad +N\bE \int^\tau_0\left(  \left\| \frac{1}{\delta_\varepsilon(t)} (f-\varepsilon \Delta u) (t,\cdot)\right\|^p_{L_p}+
							\notag
 \left\| \frac{1}{\sqrt{\delta_{\varepsilon}(t)}}  g_x  (t,\cdot)\right\|^p_{L_p}\right) \delta_{\varepsilon}(t)dt\\ 
							\notag
& \leq N \|u_0\|_{\dot \bB_p^{2 \left(1-1/ p \right)}}  \\
							\label{2019011101}
& \qquad + N\left( \|f\|^p_{\bL_p(\tau,\delta^{1-p})} + \|g\|^p_{\bL_p(\tau,\delta^{1-p/2},l_2)}
 +\bE \int^\tau_0\left(  \left\|   \Delta u (t,\cdot)\right\|^p_{L_p}\right)  \varepsilon^p \delta^{1-p}_{\varepsilon}(t)dt \right).
 \end{align}
Observe that 
$$
\varepsilon^p \delta^{1-p}_{\varepsilon}(t) \leq \varepsilon \leq 1
$$
and recall $u_{xx}\in  \bL_p( \tau) $.
Thus as $\varepsilon \downarrow 0$, we have
$$
\bE \int^\tau_0\left\|   \Delta u (t,\cdot)\right\|^p_{L_p}  \varepsilon^p \delta^{1-p}_{\varepsilon}(t)dt  \to 0.
$$

Therefore by Fatou's lemma and \eqref{2019011101}, we finally obtain  \eqref{2019011110}.

\smallskip
{\bf{Step 3}}. (General case). Finally we remove condition \eqref{2019011310} in this step.
Consider the mollification $u^\varepsilon(t,x)$ used in the proof of Theorem \ref{well thm}.
Then $u^\varepsilon$ satisfies 
\begin{align*}
& du^\varepsilon= \left( a^{ij}(t)u^\varepsilon_{x^ix^j} + f^\varepsilon(t,x ) \right)dt + (g^\varepsilon)^k(t, x) dw^k_t,  \quad  0<t \leq \tau, \\
&u^\varepsilon(0,x)=u^\varepsilon_0(x).
\end{align*}
Since $\delta(t)$ is the smallest eigenvalue of $a^{ij}(t)$,  by applying Young's convolution inequality and \eqref{2019012610},
\begin{align*}
\int_0^\tau \|u^\varepsilon_{x^ix^j}(s,\cdot)\|^p_{L_p} \delta(s) ds 
&\leq \int_0^ \tau \|u(s,\cdot)\|^p_{L_p} \|\phi^\varepsilon_{x^ix^j}\|^p_{L_1} \delta(s) ds \\
&\leq \sup_{s \leq \tau} \|u(s,\cdot)\|^p_{L_p} \|\phi^\varepsilon_{x^ix^j}\|^p_{L_1} \int_0^ \tau   a^{11}(s)ds  < \infty \quad (a.s.)
\end{align*}
for all $i$, $j$. Similarly,
$$
\|u^{\varepsilon}_{xx}\|_{\bL_p(\tau \wedge T)}<\infty
$$
for all $T>0$. 
Then denoting 
$$
\tau_n = \inf\left\{ s < \tau \wedge T :  \sum_{i,j}\int_0^s \|u^\varepsilon_{x^ix^j}(s,\cdot)\|^p_{L_p} \delta(s) ds  > n \right\},
$$
we have
$$
u^\varepsilon_{xx} \in \bL_p(  \tau_n ,\delta) \cap \bL_p(\tau_n).
$$
and $\lim_{n \to \infty}\tau_n = \tau$ $(a.s.)$. 
Thus by Step 2, for all $n \in \bN$, $\varepsilon_1, \varepsilon_2>0$,
\begin{align*}
&\|u^{\varepsilon_1}_{xx}  -u^{\varepsilon_2}_{xx} \|_{\bL_p(\tau_n ,\delta)}    \\
&\leq N(d,p) \left(  \|u_0^{\varepsilon_1}- u_0^{\varepsilon_2}\|_{\dot \bB_p^{2 \left(1-1/ p \right)}}  +  \|  f^{\varepsilon_1} - f^{\varepsilon_2}\|_{\bL_p( \tau,\delta^{1-p} )}  + \|  g^{\varepsilon_1}_x   -  g^{\varepsilon_2}_x\|_{\bL_p( \tau,\delta^{1-p/2},l_2)} \right),
\end{align*}
and
\begin{align*}
\|u^{\varepsilon_1}_{xx}\|_{\bL_p(\tau_n ,\delta)}  \leq N(d,p) \left(  \|u_0\|_{\dot \bB_p^{2 \left(1-1/ p \right)}}  +  \|  f\|_{\bL_p( \tau,\delta^{1-p} )}  + \|  g_x\|_{\bL_p( \tau,\delta^{1-p/2},l_2)} \right).
\end{align*}
Finally by Fatou's lemma and the approximation argument, we obtain \eqref{2019011110}.
The lemma is proved.
 \end{proof}

\mysection{Proof of Theorem \ref{main thm}}
								\label{pf main thm}

Since $(1-\Delta)^{\gamma/2}$ is an isometry both on Sobolev spaces and Besov spaces, 
we may assume that $\gamma=0$.

First observe that for any $\phi \in C_c^\infty(\fR^d)$ and $u \in L_p \left( \Omega, \rF ; C\left( [0,\tau ] ; L_p \right) \right)$,
\begin{align*}
&\int_0^\tau \sum_{k=1}^\infty \left|\sigma^{ik}(t) \int_{\fR^d}   u(t,x) \phi_{x^i}(x) dx \right|^2dt \\
&\leq  N(d)\int_0^\tau |\sigma|^2 (t)  \| u(t, \cdot)  \|^2_{L_p} \| \phi_x\|^2_{L_q}  dt \\
&\leq  N(d)\| \phi_x \|^2_{L_q}  \int_0^\tau |\sigma(t)|^2dt  \left( \sup_{t \in [0,\tau]} \| u(t, \cdot)  \|_{L_p(l_2)} \right)^2 < \infty \qquad (a.s.),
\end{align*}
where $q= \frac{p}{p-1}$.
Thus 
$$
 \int_0^t \sigma^{ik} (u_{x^i} ,\phi) dw_t^k
:=  -\int_0^t \sigma^{ik} (u,\phi_{x^i} ) dw_t^k
 $$
 is well-defined. 
Denote
 $$x_t^i = \int_0^t \sigma^{ik}(s)dw_s^k, \quad x_t=(x^1_t, \ldots, x_t^d).
 $$
By the It\^o-Wentzell formula (cf. \cite[Theorem 1.1]{Krylov2011}),
 $u(t,x)$ is a solution to \eqref{main eqn} if and only if $v(t,x)=u(t,x-x_t)$ is a solution to the equation
\begin{align}
							\notag
dv=& \left( \alpha^{ij}(t)v_{x^ix^j} + f(t,x - x_t) - g^k_{x^i}(t,x- x_t) \sigma^{ik}(t)  \right)dt \\
&+ g^k(t, x-x_t) dw^k_t, \quad 0<t\leq \tau; \quad v(0,x)=u_0.
							\label{2019011020}
\end{align}
By the assumption that $g_x \in \bL_p( \tau , |\sigma|^p ,l_2) \cap \bL_p( \tau , |\sigma|^p \delta^{1-p} , l_2)$,
\begin{align*}
&\bE \int_0^\tau  \int_{\fR^d}| g^k_{x^i}(t,x- x_t) \sigma^{ik}(t)|^p dx dt + \bE \int_0^\tau  \int_{\fR^d}| g^k_{x^i}(t,x- x_t) \sigma^{ik}(t)|^p \delta^{1-p} dx dt \\
&\leq \bE \int_0^\tau  \int_{\fR^d}  |g_x|_{l_2}^p(t,x)  dx |\sigma(t)|^p dt  + \bE \int_0^\tau  \int_{\fR^d}  |g_x|_{l_2}^p(t,x)  dx |\sigma(t)|^p \delta^{1-p} dt 
<\infty.
\end{align*}
Thus 
$$
g^k_{x^i}(t,x- x_t) \sigma^{ik}(t)  \in \bL_p(\tau) \cap \bL_p(\tau, \delta^{1-p}),
$$
and by Theorem \ref{thm 20181222} and Lemma \ref{2019011301}, there exists a unique solution $v$ to \eqref{2019011020} such that 
$$
v \in L_p \left( \Omega, \rF ; C\left( [0, \tau] ; L_p \right) \right),
$$
\begin{align*}
\bE \sup_{t  \in [0, \tau] } \|v(t,\cdot)\|^p_{L_p}
\leq N(p,T) \left(\|f\|^p_{\bL_p(\tau)} 
+ \|g\|^p_{\bL_p(\tau,l_2)}+\|g_x\|^p_{\bL_p( \tau, |\sigma|^p,l_2)} +\bE\|u_0\|^p_{L_p} \right),
\end{align*}
and
\begin{align*}
&\|v_{xx}\|_{\bL_p(\tau,\delta)}   \\
&\leq N(d,p) \left(  \|u_0\|_{\bB_p^{2 \left(1-1/ p \right)}}  +  \|  f\|_{\bL_p( \tau,\delta^{1-p} )} +\|g_x\|_{\bL_p( \tau, |\sigma|^p \delta^{1-p},l_2)}  + \|  g_x\|_{\bL_p( \tau,\delta^{1-p/2},l_2)} \right).
\end{align*}
Therefore, 
$$
u(t,x) := v(t,x+x_t) \in L_p \left( \Omega, \rF ; C\left( [0,\tau] ; L_p \right) \right)
$$
becomes a unique solution to equation \eqref{main eqn} and  satisfies \eqref{2019011031} and \eqref{2019011320}.
The theorem is proved. \qed

\end{document}